\documentclass[twoside]{article}
\usepackage{BSPMsample}
\usepackage{lmodern}
\usepackage[numbers]{natbib}
\usepackage[usenames,dvipsnames]{color}
\usepackage{float}
\usepackage{authblk}
\usepackage{subfigure}
\usepackage{amsmath,amssymb,amsthm,amsfonts}
\usepackage{latexsym}
\usepackage{graphicx}
\usepackage{graphics}
\usepackage{hyperref}
\usepackage[utf8]{inputenc}

\DeclareUnicodeCharacter{2217}{\textendash}

\title[{Analysis and Control of Hepatitis B Virus Spread}]{\textbf{Analysis and Control of Hepatitis B Virus Spread with Harmonic Mean Incidence and Vertical Transmission: A Compartmental Modeling Approach}} 

\author{Mustaq Ahmad, Archana Singh Bhadauria and Rachana Pathak}

\address{Mustaq Ahmad, \\Department of Mathematics and Statistics, \\Deen Dayal Upadhyaya Gorakhpur University, Gorakhpur, \\India.}
\email{mustaq.student@ddugu.ac.in}

\address{Archana Singh Bhadauria, \\Department of Mathematics and Statistics, \\Deen Dayal Upadhyaya Gorakhpur University, Gorakhpur, \\India.}
\email{archana.mathstat@ddugu.ac.in}

\address{Rachana Pathak, \\Department of Applied Science and Humanities(Mathematics), \\Faculty of Engineering \& Technology, University of Lucknow, \\India.}
\email{	rachanapathak2@gmail.com}
\date{}

\begin{document}
\maketitle

\begin{abstract}
  Hepatitis B (HBV) is a potentially fatal liver infection that remains a primary widespread disease around the globe, despite the availability of vaccinations, owing to ongoing obstacles such as vertical transmission and delayed treatment. In the present study, a compartmental mathematical model that accounts for the effects of vertical transmission and harmonic mean type incidence is put forward for the analysis of the transmission dynamics of HBV. The model’s threshold behavior is investigated through the basic reproduction number \(\mathcal{R}_0\), and the qualitative dynamics around the disease-free and endemic equilibria are examined to establish conditions for their asymptotic local and global stability. To identify key epidemiological drivers, the global sensitivity analysis of the system is performed through the incorporation of Latin Hypercube Sampling (LHS) and Partial Rank Correlation Coefficient (PRCC) methodologies. We formulate the optimal control problem using Pontryagin's maximum principle. Effective strategies for diminishing the prevalence of HBV are emphasized through numerical simulations. Numerical simulations demonstrated that optimized control measures play a significant role in reducing the number of infected individuals and increasing the recovered population. In addition, we have characterized three distinct strategies for reducing infection and observed that each has limitations in eliminating infections. However, the simultaneous administration of treatment and vaccination is useful and effective for the reduction of HBV infection.
\end{abstract}


\tableofcontents

\2010mathclass{34D20, 92D30, 49J15.}

\section{Introduction}Hepatitis B virus (HBV) is a substantial global health burden, impacting millions and resulting in thousands of cases of life-threatening liver disease annually. HBV is classified as a hepatotropic DNA virus that can cause a variety of liver-related complications, from acute hepatitis B to chronic infections that may progress to cirrhosis or hepatocellular carcinoma (HCC). The spectrum of HBV infection includes not only overt cases but also occult hepatitis B infection (OBI), where individuals may present low HBV DNA levels despite being hepatitis B surface antigen (HBsAg) negative. Such cases have the potential to reactivate, leading to acute hepatitis B. HBV transmission primarily occurs through percutaneous or mucosal exposure to infectious blood or body fluids. This transmission occurs via routes including sexual contact, injection drug use, and perinatal exposure, with the latter being the most significant contributor to chronic infections. The prevalence of HBV infection varies widely by geographical region, and chronic infection is notably more likely to develop in individuals infected during infancy, with 80\%-90\% of such cases leading to chronicity. Chronic HBV infection poses an increased risk for serious liver complications, including cirrhosis and liver cancer, necessitating ongoing medical care. Approximately 25\% of individuals who become chronically infected during childhood and 15\% of those infected later will die prematurely from liver-related diseases. Moreover, chronic infection can be influenced by host factors, including genetic variations such as single-nucleotide polymorphisms (SNPs) that have been linked to various clinical outcomes related to HBV, including spontaneous clearance and reactivation. Infectious disease models had been at the core of scientific studies in the last century, and researchers have used them for various diseases, such as tuberculosis, human immunodeficiency virus (HIV), hepatitis B, influenza, and the recent one COVID-19 \cite{castillo2002computation,diekmann2000mathematical,ghosh2019qualitative,hethcote2000mathematics,van2002reproduction}. Hepatitis B virus (HBV) infection is a disease of global health. According to the data of the World Health Organization (WHO), more than 2 billion people have been infected with HBV at some time in their lives. Of these, about 350 million remain infected chronically and become carriers of the virus. Every year, there are over 4 million acute clinical cases of HBV, and about 25 percent of carriers. HBV-related conditions, including chronic hepatitis, cirrhosis, and hepatocellular carcinoma, are responsible for more than a million deaths annually (WHO, 204) \cite{who2024global}. In 2022, only 13 percent of those living with chronic hepatitis B had been clinically diagnosed, and 3\% had received proper treatment. Achieving the elimination targets, which aim for 90\% diagnostic coverage and 80\% treatment coverage, will require major improvements in testing strategies, simplified clinical pathways, and expanded access to care. Infectious diseases of this nature have long posed significant challenges to global public health.

\vspace{0.5cm}
Numerous epidemic mathematical models have been developed by mathematicians and biologists to study the transmission dynamics of infectious diseases within populations \cite{khan2016classification, khan2017transmission,liu2022numerical,martin1974logarithmic,zaman2008stability,zarin2021fractional,zarin2021dynamics}. Several modeling approaches have been proposed specifically for Hepatitis B dynamics \cite{thornley2008hepatitis,zou2010modeling}. Fundamental component in such mathematical models is the incidence rate, as it tells us how susceptible and infectious populations interact. For infections that spread through close physical contact, including HIV and HBV etc., nonlinear incidence functions often provide a more realistic framework than traditional bilinear forms \cite{khan2021modeling,li2002qualitative, umdekar2023seir}. The interactions between HBV and host immune responses are critical, as they are not only responsible for liver injury but also play a role in the progression to cirrhosis and hepatocellular carcinoma (HCC) when the virus persists. The study of HBV transmission dynamics with a harmonic-mean-type incidence function, incorporating vertical transmission, provides a nuanced framework for modeling the spread and control of disease dynamics. This approach is often used to account for scenarios in which the transmission rate does not increase linearly with population size, potentially reflecting limitations in the transmission rate due to factors such as saturation effects in high-prevalence populations. The harmonic mean uses the reciprocals of incidence rates, giving greater weight to lower rates compared to larger ones. This approach ensures that lower incidence rates, which may better reflect the transmission dynamics of a disease during specific periods, have a more significant impact on the overall harmonic mean. Integrating the harmonic mean incidence rate into epidemic models provides a more precise depiction of temporal variations in disease transmission. This refinement enhances model realism and improves its effectiveness in analyzing and forecasting epidemic dynamics.    

\vspace{0.5cm}
Treatment and vaccination help to reduce the infection in the population and eventually to curb the epidemic in the population. But in reality, medical aids such as drugs, vaccines, therapies, hospital beds, counseling, etc., are always scarce. Now, to develop a better-performing mathematical model that can replicate the real-world phenomena and predict future trends of the epidemic in the population, many authors have applied treatment functions to their mathematical models. In order to suffice the limited medical resources, Wang \& Ruan in \cite{wang2004bifurcations} introduced the linear treatment function T(I), where \(T(I) = \begin{cases}   
r, & \text{if } I > 0 \\   
0, & \text{if } I = 0 \\      
\end{cases} \) and later, Wang in \cite{wang2006backward} redefined this idea by introducing a piecewise linear treatment function, \(T(I) = \begin{cases}   
k I, & \text{if } 0 \leq I \leq I_0\\   
k I_0, & \text{if } I \geq I_0 \\      
\end{cases} \). Zhang and Liu \cite{zhang2008backward} introduced a continuously differentiable treatment function \(T(I) = \frac{r I}{1 + \alpha I} \) to describe the saturation phenomenon of the limited medical resources. Here, T(I) is an increasing function of I, and \(\frac{r}{\alpha}\) denotes the maximal supply of medical supplies per unit time. Similarly, the Holling type II treatment function is a nonlinear response function often used to model the effects of treatment or therapy in epidemiological studies \cite{ghosh2019qualitative,GUMEL2003409}. This function is characterized by saturation, reflecting the reality that treatment efficacy cannot increase indefinitely with disease progression. Its application in modeling Hepatitis B dynamics has garnered attention due to its biological plausibility and ability to capture complex dynamics. Jayanta Kumar Ghosh et al. \cite{ghosh2019qualitative} have used the modified treatment function \(T(I,u) = \frac{r u I}{1 + b u I} \). This function increases with $I$ and $u$, and the maximum availability of medical resources is given by \(\frac{r}{b}\), where b represents the delayed treatment parameter, with r denoting the cure rate, and T decreases as b increases. It is essential to note that our study differs from some of the other relevant papers referenced in this publication because we have incorporated harmonic mean incidence along with a modified Holling type II treatment function in our epidemiological model and employed various approaches provided in the works mentioned \cite{ghosh2019qualitative,khan2020stability,khan2017transmission,zarin2021dynamics} to assess the stability or instability of the equilibrium points. In addition to that, we have used the LHS/PRCC technique to identify the parameters which are influencing most the transmission dynamics of disease in the model. It should be emphasized that this study will address both the most effective strategy to control the diseases and the qualitative analysis of the model. In addition, numerical simulations are conducted in different scenarios, and different strategies are developed to assess the beneficial effects of control measures.       

\section{Mathematical Model for the Dynamics of HBV Transmission}

In this work, we develop an SACR-type compartmental framework that integrates vaccination, a modified Holling type II treatment mechanism, and a harmonic-mean incidence structure, incorporating vertical transmission. This formulation captures both direct transmission through contact and the contribution of chronically infected mothers to new infections during childbirth. In addition, the model accounts for the influence of latent (acute) infections on overall infectiousness.

The total population at time $t$ is denoted by \[N(t)=S(t)+A(t)+C(t)+R(t),\]
and is divided into the following epidemiological groups:

\begin{itemize}
    \item \textbf{Susceptible ($S$):} individuals who can get the infection.
    \item \textbf{Acute or latent ($A$):} newly infected people who are not fully infectious yet.
    \item \textbf{Chronic ($C$):} long-term infected individuals who can spread the virus.
    \item \textbf{Recovered ($R$):} individuals who have immunity due to vaccination or recovery.
\end{itemize}

We make the following assumptions to build the model:

\begin{enumerate}
    \item Birth and natural death rates are constant. In the absence of infection, they balance each other.
    \item People mix uniformly in the population, so every individual has an equal chance of contacting others.
    \item Infection occurs at rate $\beta$. A susceptible person may get infected after contact with someone in the $A$ or $C$ class.
    \item Susceptible individuals receive vaccination at a rate $\delta$, and vaccinated persons move into the recovered class.
    \item Acute infections are treated using a Holling type~II function. Here, $u$ represents the treatment effort, $r$ is the maximum cure rate, and $b$ controls the saturation in treatment. Successfully treated acute individuals enter the recovered class.
    \item Acute individuals progress to the chronic class at a rate $\sigma$, so the average duration in the acute class is $1/\sigma$.
    \item Mothers in the chronic class may pass the infection to their newborns. This vertical transmission adds newborns directly into the acute class.
\end{enumerate}

These assumptions help us capture the main features of HBV transmission. 
The harmonic-mean incidence reflects changes in transmission pressure when the number of infectious individuals varies. 
The treatment function accounts for real-life limits in medical resources. 
Together, these components make the model more realistic and useful for studying the spread of HBV and improving control strategies.

\vspace{0.5cm}
The SACR model describes the transmission of the disease through four coupled nonlinear ordinary differential equations, which are given below:

\begin{align}
\frac{dS}{dt} &= \lambda(1-\alpha C) - \beta \frac{2AS}{A+S} - \beta \gamma \frac{2CS}{C+S} - (\delta + \mu)S \\
\frac{dA}{dt} &= \beta \frac{2AS}{A+S} + \beta \gamma \frac{2CS}{C+S} - \frac{ruA}{1+buA} - (\sigma + \mu)A \\
\frac{dC}{dt} &= \sigma A + (\alpha \lambda - \theta - \tau - \mu)C \\
\frac{dR}{dt} &= \delta S + \frac{ruA}{1+buA} + \theta C - \mu R
\end{align}

The associated initial conditions are described as:
\[S(0) \geq 0, A(0) \geq 0, C(0) \geq 0 \text{ and } R(0) \geq 0.\]

where \( \lambda \) is the recruitment rate, \( \mu \) is the rate of natural death, \( \beta \) is the rate of transmission of infection, \( \sigma \) is the progression rate from acute to chronic, and \( \theta \) is the treatment rate. The flow chart of the above model is

\hspace{-0.5cm}
\begin{figure}[ht]
    \centering
    \includegraphics[scale=0.4]{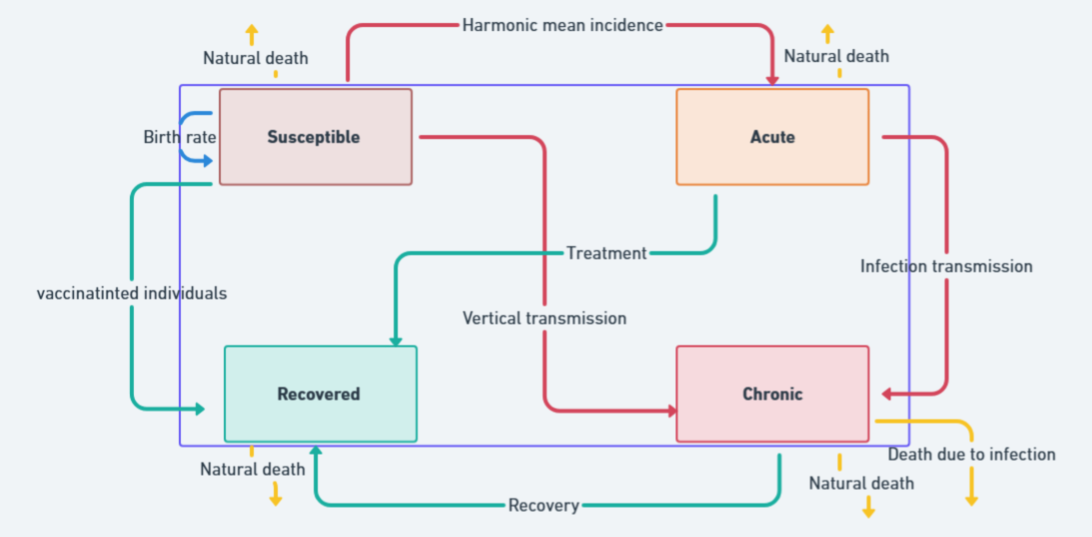}

    \caption{Schematic diagram showing the transmission process of HBV in the SACR model}
    \label{Fig_1}
\end{figure}

\begin{table}[h]
  \caption{Parameter values and their biological meanings used in the HBV model}
  \label{Table}
  \begin{center}
   \begin{tabular}{|l||l|l|}\hline 

        Parameter & Description & Value \\ \hline\hline
        $\lambda$ & Recruitment rate of susceptible population & 6.56 person $day^{-1}$ \\ \hline
        $\alpha$ & Percentage of prenatal period infected population & 0.003 $day^{-1}$\\ \hline 
        $\beta$ & Transmission rate of susceptible to infected class & 0.0095 $day^{-1}$\\ \hline
        $\gamma$ & Reduction factor in the infection transmission rate of chronic cases & 0.42 $day^{-1}$ \\ \hline
        $\delta$ & Rate of vaccination of susceptible class & 0.02 $day^{-1}$\\ \hline
        $\mu$ & Natural death rate & 0.0221 $day^{-1}$ \\ \hline
        $\sigma$ & Rate of transmission from acute to chronic class & 0.33 $day^{-1}$\\ \hline
        $\theta$ & Rate of treatment of chronic class & 0.012 $day^{-1}$ \\ \hline
        $\tau$ & Death due to chronic infection & 0.0026 $day^{-1}$\\ \hline
        $r$ & Cure rate of acutely infected population due to treatment & 1.5 person $day^{-1}$  \\ \hline
        $u$ & Rate of treatment of acute class & 0.39 $person^{-1}$ \\ \hline
        $b$ &  Saturation constant as delayed parameter of treatment & 0.5 \\ \hline
    \end{tabular}  
  \end{center}       
\end{table}

\subsection{Positivity and Boundedness of Model}

The total population N(t) of the system given above is:
\begin{align*}
N(t) &= S(t) + A(t) + C(t) + R(t), \\
\Rightarrow \frac{dN}{dt} &= \frac{dS}{dt} + \frac{dA}{dt} + \frac{dC}{dt} + \frac{dR}{dt}, \\
\Rightarrow \frac{dN}{dt} &= \lambda - \mu N - \tau C, \\
\Rightarrow \frac{dN}{dt} &\leq \lambda - \mu N.
\end{align*}

Thus, we have:
\[
N \leq \frac{\lambda}{\mu}.
\]

Assume the biologically feasible region $\Omega$ as under
\[
\Omega = \{ S, A, C, R \}  : S + A + C + R = N \leq \frac{\lambda}{\mu}, \text{where}, \, S \geq 0, \, A \geq 0, \, C \geq 0, \, R \geq 0
\}.
\]

\vspace{0.5 cm}
The above expression implies that the domain $\Omega$ remains positively invariant. This property confirms the well-defined nature of the model and its relevance from both biological and epidemiological perspectives \cite{khan2018spreading}.

\section{Disease Free Equilibrium Point (DFE)}

By equating all equations of the system (1) to zero and substituting \( A = C = 0 \), we obtain DFE \( E^0 \) as:
\[
E^0 = (S^0, A^0, C^0, R^0)
\]
\[
E^0 = \left( \frac{\lambda}{\delta + \mu}, 0, 0, \frac{\delta \lambda}{\mu(\delta + \mu)} \right).
\]

\section{Calculation of the Reproduction Number \texorpdfstring{\( R_0 \)}{R0}}
The reproduction number \cite{diekmann2000mathematical,diekmann1990definition,van2002reproduction}, denoted by $\mathcal{R}_0$, represents the level of virus transmission or spread. The steps to determine $\mathcal{R}_0$ provided by P. van den Driessche and James Watmough \cite{van2002reproduction} are as follows:

The following pair of differential equations gives the rate of change of the infected subpopulation over time:
\begin{align*}
    \frac{dA}{dt} &= \beta \frac{2AS}{A+S} + \beta \gamma \frac{2CS}{C+S} - \frac{ruA}{1+buA} - (\sigma + u)A \\
    \frac{dC}{dt} &= \sigma A + (\alpha \lambda - \theta - \tau - \mu)C \\ 
\end{align*}

The above equations can be expressed in vector form as follows:
\begin{align*}
    \frac{d}{dt}
    \begin{pmatrix}
        x
    \end{pmatrix}
    &= \mathcal{F} - \mathcal{V},
\end{align*}
where 
\[
    \mathbf{x} = \begin{pmatrix} A \\ C \end{pmatrix}, \quad
    \mathcal{F} = \begin{pmatrix}\beta \frac{2AS}{A+S} + \beta \gamma \frac{2CS}{C+S} \\ 0 \end{pmatrix}, \quad
    \mathcal{V} = \begin{pmatrix}\frac{ruA}{1+buA} + (\sigma + \mu)A \\ (\theta + \tau + \mu - \alpha \lambda)C - \sigma A \end{pmatrix}.
\]

The Jacobian matrices obtained from the linearization of the above equations are:
\[
    D\mathcal{F} = \begin{pmatrix} 2 \beta \left(\frac{S}{A+S} \right)^2  & 2 \beta \gamma \left(\frac{S}{C+S} \right)^2  \\ 0 & 0 \end{pmatrix}, \quad
    D\mathcal{V} = \begin{pmatrix} \frac{r u }{(1 + b u A)^2} 
    + \sigma + \mu & 0 \\ -\sigma & \theta + \tau + \mu - \alpha \lambda \end{pmatrix}.
\]

Evaluating these Jacobian matrices at the disease-free equilibrium (DFE) point $E^0$, we get:
\[
    F = D\mathcal{F} \big|_{E^0} = \begin{pmatrix} 2 \beta & 2 \beta \gamma \\ 0 & 0 \end{pmatrix}, \quad
    V = D\mathcal{V} \big|_{E^0} = \begin{pmatrix} r u + \sigma + \mu  & 0 \\ -\sigma & \theta + \tau + \mu - \alpha \lambda\end{pmatrix}.
\]

The next generation matrix $\mathcal{G}$ is given by:
\[
    \mathcal{G} = F V^{-1} = \begin{pmatrix} \frac{2 \beta (\theta + \tau + \mu - \alpha \lambda + \gamma \sigma)}{(r u + \sigma + \mu)(\theta + \tau + \mu - \alpha \lambda)} & 0 \\ 0 & 0 \end{pmatrix}.
\]

The eigenvalues of the matrix $\mathcal{G}$ are found  by solving the characteristic equation:
\[
    |\mathcal{G} - x I| = 0.
\]

The eigenvalues are $x_1 = 0$ and $x_2 = \frac{2 \beta (\theta + \tau + \mu - \alpha \lambda + \gamma \sigma)}{(r u + \sigma + \mu)(\theta + \tau + \mu - \alpha \lambda)}$.
\vspace{0.5in}

The reproduction number, $\mathcal{R}_0$, is obtained as the leading eigenvalue of the matrix $\mathcal{G}$:

\[
    \mathcal{R}_0 = \frac{2 \beta (\theta + \tau + \mu - \alpha \lambda + \gamma \sigma)}{(r u + \sigma + \mu)(\theta + \tau + \mu - \alpha \lambda)}.
\]

Thus, the reproduction number is:
\begin{align*}
\mathcal{R}_0 &= \frac{2 \beta}{r u + \sigma + \mu} + \frac{2 \beta \sigma \gamma}{(r u + \sigma + \mu)(\theta + \tau + \mu - \alpha \lambda)} \\
\mathcal{R}_0 &= {\mathcal{R}_0}_A + {\mathcal{R}_0}_C
\end{align*}
where, \({\mathcal{R}_0}_A = \frac{2 \beta}{r u + \sigma + \mu} \) and \({\mathcal{R}_0}_C = \frac{2 \beta \sigma \gamma}{(r u + \sigma + \mu)(\theta + \tau + \mu - \alpha \lambda)}\).

\section{Existence of Endemic Equilibrium: Isocline Method}In this section, we put to use the isocline method \cite{yadav2023study} to prove the existence of an endemic equilibrium($E^*$), where $E^*$ is the solution to a system of equations
\begin{align}
\lambda(1-\alpha C) - \beta \frac{2AS}{A+S} - \beta \gamma \frac{2CS}{C+S} - (\delta + \mu)S = 0 \\
\beta \frac{2AS}{A+S} + \beta \gamma \frac{2CS}{C+S} - \frac{ruA}{1+buA} - (\sigma + \mu)A = 0 \\
\sigma A + (\alpha \lambda - \theta - \tau - \mu)C = 0 \\
\delta S + \frac{ruA}{1+buA} + \theta C - \mu R = 0
\end{align}

From the above equation(7), we have 
\[
A = \frac{1}{\sigma} (\theta + \tau + \mu - \alpha \lambda) C = \frac{\gamma_2 C}{\sigma}
\]
where \(\gamma_2 = \theta + \tau + \mu - \alpha \lambda \).

From equation (5), we obtain

\[
\chi(S,C) = \lambda (1 - \alpha C) - 2 \beta \left( \frac{\gamma_2 C S}{\gamma_2 C + \sigma S} + \frac{\gamma S C}{S + C} \right) - (\delta + \mu) S = 0 \tag{9}
\]

From equation (6), we obtain

\[
\eta(S,C) = 2 \beta \left( \frac{\gamma_2 C S}{\gamma_2 C + \sigma S} +  \frac{\gamma S C}{S + C} \right) - \frac{(\sigma + \mu) \gamma_2 C}{\sigma} - \frac{r u \gamma_2 C}{\sigma + b u \gamma_2 C} = 0 \tag{10}
\]

Substituting \( C = 0 \) in equations (9) and (10), we get \(\chi(S_a, 0) = 0\) and \(\eta(S_b, 0) = 0\), and thus we obtain 
\[
S_a = \frac{\lambda}{\delta + \mu}, \quad S_b = \frac{1}{{R_0}}
\]

For equation (9), with $C \neq 0$, we evaluate
\[
\frac{dS}{dC} = -\frac{\partial \chi / \partial C} {\partial \chi /\partial S} = \frac{\alpha\lambda + 2\beta \left( \frac{\sigma \gamma_2}{(\gamma_2 C + \sigma S)^2} + \frac{\gamma}{(S + C)^2} \right) S^2}{2 \beta \left( \left( \frac{\gamma_2}{\gamma_2 C + \sigma S} \right)^2 + \frac{\gamma}{(S + C)^2} \right) C^2 + (\delta + \mu)}> 0.
\]
Thus, we can conclude that the isocline given by equation (9) is a monotonically increasing function of C.

\vspace{0.5cm}
Further from equation (10), with $C \neq 0$, we evaluate 

\[ \frac{dS}{dC} = - \frac{\partial \eta / \partial C}{\partial \eta / \partial S} = - \left[ \frac{2 \beta \left( \frac{\sigma \gamma_2 S^2}{(\gamma_2 C + \sigma S)^2} + \frac{\gamma S^2}{(S + C)^2} \right) - \left(\frac{(\sigma + \mu) \gamma_2}{\sigma} + \frac{\sigma r u \gamma_2}{(\sigma + b u \gamma_2 C)^2} \right)}{2 \beta \left( \left( \frac{\gamma_2 C}{\gamma_2 C + \sigma S} \right)^2 + \frac{\gamma C^2}{(S + C)^2} \right)} \right] \]

Therefore, \(\frac{dS}{dC} < 0\) if

\[
2 \beta \left( \frac{\sigma \gamma_2 S^2}{(\gamma_2 C + \sigma S)^2} + \frac{\gamma S^2}{(S + C)^2} \right) > \left(\frac{(\sigma + \mu) \gamma_2}{\sigma} + \frac{\sigma r u \gamma_2}{(\sigma + b u \gamma_2 C)^2} \right)
\tag{11}
\]
Here, we can conclude that the isocline given by equation (10) is a monotonically increasing function of C if equation (11) holds. These conditions suggest that the two isoclines (9) and (10) intersect at a single, unique point where the above inequalities are satisfied. Thus, from the above discussion, we can infer that the endemic equilibrium exists if the conditions mentioned above hold.

\section{Stability Analysis of the Mathematical Model}

 In the study of dynamical systems, stability can be visualized geometrically by examining how trajectories in a phase space (or state space) behave over time in relation to equilibrium points. Equilibrium points (or fixed points) represent steady states in which the system remains if it is not perturbed. For an epidemiological model, these could represent steady states of infection levels, such as a disease-free equilibrium (DFE) or an endemic equilibrium (EE).

\subsection{Local Stability Analysis of HBV Mathematical Model}
Local stability around an equilibrium point means that trajectories starting near the equilibrium will move back toward it if slightly perturbed. Geometrically, in phase space, this means that nearby trajectories spiral or converge toward the equilibrium. To investigate the local stability around both equilibrium points (the disease-free equilibrium and the endemic equilibrium), we determine the Jacobian matrix and analyze its eigenvalues.

\newtheorem{theorem}{Theorem}
\begin{theorem} If \({\mathcal{R}_0} < 1 \), then the DFE \(E^0\) of the system (1) is locally asymptotically stable.
\end{theorem}

\begin{proof} The Jacobian matrix of the model at DFE \( E^0 = \left( \frac{\lambda}{\delta + \mu}, 0, 0, \frac{\lambda \delta}{\mu(\delta + \mu)} \right) \) is

\[
J_{E^0} = 
\begin{bmatrix}
-(\delta + \mu) & -2 \beta & -\lambda \alpha - 2 \beta \gamma & 0 \\
0 & 2 \beta - \sigma - \mu - ru & 2 \beta \gamma & 0 \\
0 & \sigma & \alpha \lambda - \theta - \tau - \mu & 0 \\
\sigma & ru & \theta & -\mu
\end{bmatrix}
\]

Here, this Jacobian has eigenvalues -$\mu$, -$(\delta + \mu)$, and the remaining two eigenvalues are obtained from the reduced Jacobian matrix
\[
J_{E^0} = 
\begin{bmatrix}
2 \beta - \sigma - \mu - ru & 2 \beta \gamma \\
\sigma & \alpha \lambda - \theta - \tau - \mu \\
\end{bmatrix}
\]
whose characteristic equation is
\[
m^2 + p m + q = 0
\]
where
\begin{align*}
    p &= (\sigma + \mu + ru - 2 \beta) + (\theta + \tau + \mu -\alpha \lambda) \\
      &= (\theta + \tau + \mu -\alpha \lambda) + (\sigma + \mu + ru) (1 - \frac{2 \beta}{\sigma + \mu + ru}) \\
      &= (\theta + \tau + \mu -\alpha \lambda) + (\sigma + \mu + ru) (1 - \mathcal{R}_{0A}) \\
    q &= (\sigma + \mu + ru - 2 \beta) (\theta + \tau + \mu -\alpha \lambda) - 2 \beta \gamma \sigma \\
      &= (\sigma + \mu + ru) (1 - \mathcal{R}_{0A}) (\theta + \tau + \mu -\alpha \lambda) - 2 \beta \gamma \sigma  \\
      &= (\sigma + \mu + ru) (1 - \mathcal{R}_{0A}) (\theta + \tau + \mu -\alpha \lambda) \left( 1 - \frac{\mathcal{R}_{0C}}{1 - \mathcal{R}_{0A}} \right) \\
      &= (\sigma + \mu + ru) (1 - \mathcal{R}_{0A}) (\theta + \tau + \mu -\alpha \lambda) \left(\frac{1 - \mathcal{R}_{0}}{1 - \mathcal{R}_{0A}} \right).
\end{align*}

From the above equation, we observe that $p > 0$ if $\mathcal{R}_{0A} < 1$ and $q > 0$ if ${\mathcal{R}_0}_A < 1$ \& ${\mathcal{R}_0} < 1$.


Therefore, when ${\mathcal{R}_0} < 1$, all eigenvalues have negative real parts. 
This implies that the system is locally asymptotically stable whenever ${\mathcal{R}_0} < 1$.
\end{proof}

\begin{theorem} If $\mathcal{R}_0 > 1$ then the endemic equilibrium point $E^*$ = ($S^*, A^*, C^*, R^*$) is locally asymptotically stable.
\end{theorem}

\begin{proof}
The Jacobian of the given system is
\[
\hspace{-1cm}
J = 
\begin{bmatrix}
- \left(2 \beta \left( \frac{A^*}{A^* + S^*} \right)^2 + 2 \beta \gamma \left( \frac{C^*}{C^* + S^*} \right)^2 + (\delta + \mu) \right) & -2 \beta \left( \frac{S^*}{A^* + S^*} \right)^2 & - \left(2 \beta \gamma \left( \frac{S^*}{C^* + S^*} \right)^2 + \alpha \lambda \right) & 0 \\ 
2 \beta \left( \frac{A^*}{A^* + S^*} \right)^2 + 2 \beta \gamma \left( \frac{C^*}{C^* + S^*} \right)^2 & 2 \beta \left( \frac{S^*}{A^* + S^*} \right)^2 - \frac{r u}{(1 + b u A^*)^2} -(\sigma + \mu) & 2 \beta \gamma \left( \frac{S^*}{C^* + S^*} \right)^2 & 0 \\ 
0 & \sigma &  \alpha \lambda - \theta - \tau - \mu & 0 \\ 
\delta & \frac{r u}{(1 + b u A^*)^2} & \theta & -\mu \\ 
\end{bmatrix}
\]

It is easy to see that one eigenvalue of the Jacobian matrix $J$, evaluated at the disease–present equilibrium $E^{*}$, is negative; in fact, $x_{1} = -\mu < 0$. 
To study the remaining three eigenvalues, we consider the following reduced matrix:
\[
J_1^{|\ast|} = 
\begin{bmatrix}
-l_1 & -2 \beta P_1 & -l_2 \\ 
l_1 - q_3 & l_3 & 2 \beta \gamma P_2 \\ 
0 & \sigma & -q_2 \\ 
\end{bmatrix}
\]
		
where
\begin{align*}
    l_1 &= 2 \beta \left( \frac{A^*}{A^* + S^*} \right)^2 + 2 \beta \gamma \left( \frac{C^*}{C^* + S^*} \right)^2 + (\delta + \mu), \\
    l_2 &= \alpha \lambda + 2 \beta \gamma \left( \frac{S^*}{C^* + S^*} \right)^2, \\
     l_3 &= 2 \beta \left( \frac{S^*}{A^* + S^*} \right)^2 - \frac{r u}{(1 + b u A^*)^2} - (\sigma + \mu), \\
    q_2 &= \theta + \tau + \mu - \lambda \alpha, \\
    q_3 &= \delta + \mu, \\
    q_4 &= \sigma + \mu, \\
    P_1 &= \left( \frac{S^*}{A^* + S^*} \right)^2, \\ 
    P_2 &= \left( \frac{S^*}{C^* + S^*} \right)^2.
\end{align*}

Applying elementary row operations, we obtain
\[
J_2^{|\ast|} = 
\begin{bmatrix}
-l_1 & -2 \beta P_1 & -l_2 \\ 
0 & -Z_1 & Z_2 \\ 
0 & 0 & -Z_3 \\ 
\end{bmatrix}
\]
where,
\begin{align*}
    Z_1 &= l_1 l_3 - 2 \beta P_1 (l_1 - q_3), \\
    Z_2 &= 2 \beta \gamma P_2 l_1 - l_2 (l_1 - q_3), \\
    Z_3 &= \sigma Z_2 - q_2 Z_1 \\
\end{align*}

\text{The eigenvalues of } \( J_2^{|\ast|} \) \text{ takes the following form:}

\begin{align*}
    x_2 &= -l_1 < 0, \\
    x_3 &= -Z_1 < 0  \\
    x_4 &= -Z_3 < 0 
\end{align*}
Clearly, the second eigenvalue $x_2$ has a negative real part. Moreover, the last two eigenvalues also possess negative real components, provided that \({\mathcal{R}_0}_C >1\) and \({\mathcal{R}_0}_A < \frac{Z_2}{\gamma Z_1}\), as well as $Z_1 > 0$ and $Z_2 > 0$. Under these circumstances, all eigenvalues exhibit a negative real component, signifying that our system is locally asymptotically stable.
\end{proof}

\subsection{Global Stability Analysis of HBV Mathematical Model}Global stability implies that any trajectory, regardless of where it starts, will eventually converge to the equilibrium point. This can be visualized as a valley with a single lowest point, where every path or slope eventually leads to the same equilibrium. In this section, we used the M-matrix method to prove the global asymptotic stability of the disease-free equilibrium. Furthermore, we applied the geometric approach, which extends the Lyapunov theory \cite{li1995global}, to establish the global asymptotic stability of the endemic equilibrium.

\subsubsection{Global Stability Analysis at Disease-Free Equilibrium}This section offers a brief analysis of the Castillo-Chavez method \cite{castillo2002mathematical} to demonstrate the global stability of the model at the disease-free equilibrium (DFE) point \cite{castillo2002computation,khan2017transmission}. Following this approach, we first restructure the model into a subsystem with two equations:

\begin{align}
    \frac{d\mathcal{X}_1}{dt} &= G(\mathcal{X}_1, \mathcal{X}_2), \tag{12} \\
    \frac{d\mathcal{X}_2}{dt} &= H(\mathcal{X}_1, \mathcal{X}_2). \tag{13}
\end{align}

In the above subsystem (12-13), $\mathcal{X}_1$ and $\mathcal{X}_2$ represent the uninfected and infected populations, respectively. Here, $\mathcal{X}_1 = (S, R) \in \mathbb{R}^2$ corresponds to uninfected individuals, while $\mathcal{X}_2 = (A, C) \in \mathbb{R}^2$  includes individuals in different stages of infection, such as acute infection and chronic carrier states. The disease-free equilibrium (DFE) denoted as $E^0$, is given by $E^0 = (\mathcal{X}_1^0, 0)$. To prove the global stability of the DFE, the following two conditions must hold:

1. If $\frac{d\mathcal{X}_1}{dt} = G(\mathcal{X}_1, 0)$, then the state $\mathcal{X}_1$ is globally asymptotically stable.

2.  $H(\mathcal{X}_1, \mathcal{X}_2) = B\mathcal{X}_2 - \Tilde{H}(\mathcal{X}_1, \mathcal{X}_2)$ must hold, where $H(\mathcal{X}_1, \mathcal{X}_2) \geq 0$ for all $(\mathcal{X}_1, \mathcal{X}_2) \in \Omega$.

In this condition, $B = D_{\mathcal{X}_2}H(\mathcal{X}_1^0, 0)$ is an M matrix, which means that all off-diagonal elements are positive. The region $\Omega$ represents the feasible domain. Given these conditions, we present the following lemma:

\vspace{0.5cm}
\textbf{Lemma 1.}: For \( R_0 < 1 \), the equilibrium point \( E^0 = (\mathcal{X}_1^0, 0) \) of our system is globally asymptotically stable if the above conditions are satisfied.

\begin{theorem}
If \( R_0 < 1 \), then the model is globally asymptotically stable at DFE \( E^0 \). However, if \( R_0 > 1 \), the model is considered unstable.
\end{theorem}

\begin{proof}
We assume the uninfected and infected populations are represented by \( X_1 = (S, R) \) and \( X_2 = (A, C) \), respectively. Now we define \( E^0 = (X_1^0, 0) \), that is,
\[
X_1^0 = \left( S^0, R^0 \right) = \left( \frac{\lambda}{\delta + \mu}, \frac{\lambda \delta}{\beta + \delta + \mu} \right).
\]

By implementing the proposed model, we have
\[
\frac{dX_1}{dt} = G(X_1, X_2),
\]
and for \( S = S^0 \), \( R = R^0 \), we get
\[ \frac{dX_1^0}{dt} = G(X_1^0, 0), \]
i.e., \( X_2 = 0 \Rightarrow (A, C) = (0, 0) \)

or
\[
G(X_1^0, 0) = \begin{bmatrix} \lambda - (\delta + \mu) S^0 \\ \delta S^0 - \mu R^0 \end{bmatrix} = 0.
\]

Thus, by the theorem above as \( t \to \infty \), \( X_1 \to X_1^0 \), so \( X_1 = X_1^0 \) is stable globally asymptotically.

We now proceed to prove the second scenario, which is
\[
BX_2 - \hat{H}(X_1, X_2) = H(X_1, X_2),
\]
where, \( \hat{H}(X_1, X_2) > 0 \) for \( (X_1, X_2) \in \Omega \). Let \( BX_2 - H(X_1, X_2) = \hat{H}(X_1, X_2) \), where \( B = D_{X_2} H(X_1^0, 0) \) is a Metzler matrix (M-matrix).

\subsection*{Metzler Matrix \( B \)}
\[
B = \begin{bmatrix} 2 \beta - (\sigma + \mu + r u) & 2 \beta \gamma \\ \sigma & \alpha \lambda - (\theta + \tau + \mu) \end{bmatrix}
\]
Define the function \( \hat{H}(X_1, X_2) \) as:

\begin{align*}
\hat{H}(X_1, X_2) &= \begin{bmatrix} 
2\beta - (\sigma + \mu + r u) & 2 \beta \\
\sigma & \alpha \lambda - (\theta + \tau + \mu)
\end{bmatrix} \begin{bmatrix} A \\ C \end{bmatrix} - \begin{bmatrix}
2\beta \left(\frac{A S}{A + S} \right) + 2\beta \gamma \left(\frac{C S}{C + S} \right) - \frac{r u A}{1 + b u A} - (\sigma + \mu) A \\
(\alpha \lambda - \theta -\tau -\mu) C + \sigma A
\end{bmatrix} \\
&= \begin{bmatrix}
        2\beta \left( 1 - \frac{A S}{A + S} \right) + 2\beta \gamma \left( 1 - \frac{C S}{C + S} \right) + r u A \left( \frac{1}{1 + b u A} - 1 \right) \\
        0
\end{bmatrix}
\end{align*}

Here,
\[
0 \leq \hat{H}(X_1, X_2), \quad \text{if} \quad \frac{A S}{A + S} \leq 1, \quad \frac{C S}{C + S} \leq 1, \frac{1}{1 + b u A} \geq 1.
\]

Since conditions (1) and (2) are satisfied, it follows from Lemma 1 that the disease-free equilibrium (DFE) $E^0$ is globally asymptotically stable. Moreover, the matrix $B$ is an M-matrix.
\end{proof}

\subsubsection{Global Stability Analysis at Endemic Equilibrium point}For the global stability of our model at the endemic equilibrium point \( E^* \), the geometrical approach \cite{li1996geometric} is used. Thus, we analyze the sufficient condition under which \( E^*\) is globally asymptotically stable. Let us, consider the differential equation
\[
\dot{x} = f(x), \tag{14}
\]
where the open set \( D \subset \mathbb{R}^n \) is simply connected, and \( f : D \rightarrow \mathbb{R}^n \) is a function such that \( f \in C^1(D) \). Assuming that \( f(x^*) = 0 \) is the solution of equation (14), for \( x(t, x_0) \), the following are true:

1. There exists a compact absorbing set \( K \in D \).

2. System (14) has a unique equilibrium.

\noindent The solution \( x^* \) is considered globally asymptotically stable in \(D \), if it is locally asymptotically stable and all trajectories in \(D \) approach the equilibrium \( x^* \). For \( n \geq 2 \), a condition that prevents the existence of non-constant periodic solutions of equation (14) is known as the Bendixson criterion. The classical Bendixson criterion \( \operatorname{div} f(x) < 0 \) for \( n = 2 \) is robust under \( C^1 \) \cite{li1996geometric}. Additionally, for equation (14) a point \( x_0 \in D \) is wandering if there is a neighbourhood \( G \) of \( x_0 \) and \( \epsilon > 0 \) such that \( G \cap x(t, G) \) is empty for all \( t > \epsilon \). This establishes the following global stability principle for an autonomous system of any finite dimension.

\vspace{0.5cm}

\textbf{Lemma 2.}: If the Bendixson criterion and conditions (1)–(2) hold for equation (14) meaning that the system remains robust under \( C^1 \) local perturbation of \( f \) at all non-equilibrium, non-wandering points for equation (14)), then \( x^* \) is globally asymptotically stable in \( D \), provided it is stable.

\textit{Proof.} Define a matrix-valued function \( M \) on \( D \) by
\[
M(x) = \begin{pmatrix} n \\ 2 \end{pmatrix} \times \begin{pmatrix} n \\ 2 \end{pmatrix}. \tag{15}
\]

Furthermore, assume that \( M^{-1} \) exists and is continuous for \( x \in K \). Now define
\[
\bar{q} = \limsup_{t \to \infty} \frac{1}{t} \int_0^t \mu \left( B \left( x(s, x_0) \right) \right) ds, \tag{16}
\]
where \( J^{[2]} \) is the second additive compound matrix \cite{muldowney1990compound} of \( J \), that is, \( J(x) = Df(x) \) and \( B = M_f M^{-1} + MJ^{[2]}M^{-1} \). Let \( \mu(B) \) denote the Lozinski measure of the matrix \( B \) with respect to the norm \( \| \cdot \| \) in \( \mathbb{R}^n \) \cite{li1996geometric} defined by
\[
\mu(B) = \lim_{x \to 0} \frac{|I + Bx| - 1}{x}. \tag{17}
\]

This, if \( \bar{q} < 0 \), in this case, the existence of an orbit results in a simple closed rectifiable curve, such as a periodic orbit or a heteroclinic cycle.

\vspace{0.5cm}

\textbf{Lemma 3}: Let \( D \) be simply connected and let conditions (1)–(2) be satisfied. Then the unique equilibrium \( x^* \) of equation (14) is globally asymptotically stable in \( D \), if \( \bar{q} < 0 \).

Next, we apply these techniques to establish the global stability of the proposed model at the endemic equilibrium.

\begin{theorem}: If \( R_0 > 1 \), then the model is globally asymptotically stable at EE points \( E^* = (S^*, A^*, C^*, R^*) \) and is otherwise unstable.
\end{theorem}

\begin{proof} Jacobian matrix \( J \) of the model is:

\[
J = \begin{bmatrix}
a_{11} & a_{12} & a_{13} & a_{14} \\
a_{21} & a_{22} & a_{23} & a_{24} \\
a_{31} & a_{32} & a_{33} & a_{34} \\
a_{41} & a_{42} & a_{43} & a_{44} \\
\end{bmatrix}
\]

The third additive compound matrix is expressed as follows:

\[
J^{[3]} = \begin{bmatrix}
a_{11} + a_{22} + a_{33} & a_{34} & -a_{24} & a_{14} \\
a_{43} & a_{11} + a_{22} + a_{44} & a_{23} & -a_{13} \\
-a_{42} & a_{32} & a_{11} + a_{33} + a_{44} & a_{12} \\
-a_{41} & -a_{31} & a_{21} & a_{22} + a_{33} + a_{44} \\
\end{bmatrix}
\]

where the Jacobian of the system is:

\[
\hspace{-2cm}
J = 
\begin{bmatrix}
- \left(2 \beta \left( \frac{A^*}{A^* + S^*} \right)^2 + 2 \beta \gamma \left( \frac{C^*}{C^* + S^*} \right)^2 + q_3 \right) & -2 \beta \left( \frac{S^*}{A^* + S^*} \right)^2 & - \left( 2 \beta \gamma \left( \frac{S^*}{C^* + S^*} \right)^2 + \alpha \lambda \right) & 0 \\ 
2 \beta \left( \frac{A^*}{A^* + S^*} \right)^2 + 2 \beta \gamma \left( \frac{C^*}{C^* + S^*} \right)^2 & 2 \beta \left( \frac{S^*}{A^* + S^*} \right)^2 - \frac{r u}{(1 + b u A^*)^2} - q_4 & 2 \gamma \left( \frac{S^*}{C^* + S^*} \right)^2 & 0 \\ 
0 & \sigma & - q_2 & 0 \\ 
\delta & \frac{r u}{(1 + b u A^*)^2} & \theta & -\mu \\ 
\end{bmatrix}
\]

Corresponding additive compound matrix \( J^{[3]} \) can be written as:

\[
J^{[3]} = \begin{bmatrix}
B_{11} & 0 & 0 & 0 \\
\theta & B_{22} & 2 \gamma \left( \frac{S^*}{S^* + C^*} \right)^2 & \left( 2 \beta \gamma \left( \frac{S^*}{C^* + S^*} \right)^2 + \alpha \lambda \right)  \\
- \frac{r u}{(1 + b u A^*)^2} & \sigma & B_{33} & -2\beta \left( \frac{S^*}{S^* + A^*} \right)^2 \\
\delta & 0 & 2 \beta \left( \frac{A^*}{A^* + S^*} \right)^2 + 2\beta \gamma \left( \frac{C^*}{C^* + S^*} \right)^2 & B_{44} \\
\end{bmatrix}
\]
where,
\begin{align*}
B_{11} & = -(a_{11} + a_{22} + a_{33}) \\
      & = - \left\{ 2 \beta \left( \frac{A^*}{A^* + S^*} \right)^2 + 2 \beta \gamma \left( \frac{C^*}{C^* + S^*} \right)^2 - 2 \beta \left( \frac{S^*}{S^* + A^*} \right)^2 + \frac{r u}{(1 + b u A^*)^2} + q_2 + q_3 + q_4 \right\} \\
B_{22} & = -(a_{11} + a_{22} + a_{44}) \\
      & = - \left\{ 2 \beta \left( \frac{A^*}{A^* + S^*} \right)^2 + 2 \beta \gamma \left( \frac{C^*}{C^* + S^*} \right)^2 + \frac{r u}{(1 + b u A^*)^2} - 2 \beta \left( \frac{S^*}{A^* + S^*} \right)^2 + q_3 + q_4 + \mu \right\} \\
B_{33} & = -(a_{11} + a_{33} + a_{44}) \\
      & = - \left\{ 2 \beta \left( \frac{A^*}{A^* + S^*} \right)^2 + 2 \beta \gamma \left( \frac{C^*}{C^* + S^*} \right)^2 + q_2 + q_3 + \mu \right\} \\
B_{44} & = -(a_{22} + a_{33} + a_{44}) \\
      & = - \left\{- 2 \beta \left( \frac{S^*}{S^* + A^*} \right)^2 + \frac{r u}{(1 + b u A^*)^2} + q_2 + q_4 + \mu \right\}
\end{align*}

We assume the function \( Q = \text{diag}(S, A, C, R) \) then \( Q^{-1} = \left[ \frac{1}{S} \; \frac{1}{A} \; \frac{1}{C} \; \frac{1}{R} \right] \) and derivative of \( Q \) w.r.t time yields:

\[
Q_f = \text{diag}\left( \dot{S}, \dot{A}, \dot{C}, \dot{R} \right)
\]

\[
Q_f Q^{-1} = \text{diag} \left( \frac{\dot{S}}{S}, \frac{\dot{A}}{A}, \frac{\dot{C}}{C}, \frac{\dot{R}}{R} \right) \quad 
\]
\text{and}
\[
Q J^{[3]} Q^{-1} = \begin{bmatrix}
B_{11} & 0 & 0 & 0 \\
\frac{A^*}{S^*} \theta & B_{22} & 2 \gamma \frac{A^*}{C^*} \left( \frac{S^*}{S^* + C^*} \right)^2 & \frac{A^*}{R^*} \left( 2 \beta \gamma \left( \frac{S^*}{C^* + S^*} \right)^2 + \alpha \lambda \right)  \\
- \frac{C^*}{S^*} \frac{r u}{(1 + b u A^*)^2} & \frac{C^*}{A^*} \sigma & B_{33} & -2 \beta \frac{C^*}{R^*} \left( \frac{S^*}{S^* + A^*} \right)^2 \\
\frac{R^*}{S^*} \delta& 0 & \ 2 \beta \frac{R^*}{C^*} \left(\left( \frac{A^*}{A^* + S^*} \right)^2 +  \gamma \left( \frac{C^*}{C^* + S^*} \right)^2 \right) & B_{44} \\
\end{bmatrix}
\]

Now, a direct computation shows that \(
B = Q_f Q^{-1} + Q_f J^{[3]} Q^{-1} \), which becomes

\[
B = \begin{bmatrix}
b_{11} & 0 & 0 & 0 \\
b_{21} & b_{22} & b_{23} & b_{24} \\
b_{31} & b_{32} & b_{33} & b_{34} \\
b_{41} & 0 & b_{43} & b_{44} \\
\end{bmatrix}
\]
where,

\begin{align*}
b_{11} &= \frac{\dot{S}}{S} - \left[ 2\beta \left( \frac{A^*}{A^* + S^*} \right)^2 + 2\beta \gamma \left( \frac{C^*}{C^* + S^*} \right)^2 + \frac{r u}{(1 + b u A^*)^2}- 2\beta \left( \frac{S^*}{S^* + A^*} \right)^2 + (q_2 + q_3 + q_4) \right] \\ b_{21} &= \frac{A^*}{S^*} \theta, \\ \quad b_{22} &= \frac{\dot{A}}{A} - \left[ 2\beta \left( \frac{A^*}{A^* + S^*} \right)^2 + 2\beta \gamma \left( \frac{C^*}{C^* + S^*} \right)^2 - 2\beta \left( \frac{S^*}{S^* + A^*} \right)^2 + \frac{r u}{(1 + b u A^*)^2} + q_2 + q_4 + \mu \right], \\ b_{23} &= 2\gamma \frac{A^*}{C^*} \left( \frac{S^*}{S^* + C^*} \right)^2, \\ \quad b_{24} &= \frac{A^*}{R^*} \left( \alpha \lambda + 2\beta \left( \frac{S^*}{S^* + A^*} \right)^2 \right), \\ b_{31} &= - \frac{C^*}{S^*} \frac{r u}{(1 + b u A^*)^2} , \\ \quad b_{32} &= \frac{C^*}{A^*} \sigma, \\ b_{33} &= \frac{\dot{C}}{C} - \left[ 2\beta \left( \frac{A^*}{A^* + S^*} \right)^2 + 2\beta \gamma \left( \frac{C^*}{C^* + S^*} \right)^2 + q_2 + q_3 + \mu \right], \\ \quad b_{34} &= -2\beta \frac{C^*}{R^*} \left( \frac{S^*}{S^* + A^*} \right)^2, \\ b_{41} &= \frac{\dot{R}}{R} - \left[ -2\beta \left( \frac{S^*}{S^* + A^*} \right)^2 + \frac{r u}{(1 + b u A^*)^2} + q_2 + q_4 +\mu \right] 
\end{align*}   

Consequently,
\begin{align*}
\begin{cases}
h_1(t) &= b_{11} + \sum_{i=2}^{4} |b_{1i}| \\
h_1(t) &= b_{11} + |b_{12}| + |b_{13}| + |b_{14}| \\
       &= \frac{\dot{S}}{S} - \Bigg[ 
           2\beta \left( \frac{A^*}{A^* + S^*} \right)^2 
           + 2\beta \gamma \left( \frac{C^*}{C^* + S^*} \right)^2 
           + \frac{r u}{(1 + b u A^*)^2} + q_2 + q_3 + q_4 \\
       &\qquad - 2\beta \left( \frac{S^*}{S^* + A^*} \right)^2 
           \Bigg] \\
       &\leq \frac{\dot{S}}{S} - (q_2 + q_3 + \mu),  
\end{cases}       
\end{align*}

\begin{align*}
\begin{cases}
h_2(t) &= b_{22} + \sum_{i=1, i \neq 2}^{4} |b_{2i}| \\
h_2(t) &= b_{22} + |b_{21}| + |b_{23}| + |b_{24}| \\
       &= \frac{\dot{A}}{A} - \Bigg[ 
           2\beta \left( \ frac{A^*}{A^* + S^*} \right)^2 
           + 2\beta \gamma \left( \frac{C^*}{C^* + S^*} \right)^2 
           - 2\beta \left( \frac{S^*}{S^* + A^*} \right)^2 \\
       &\qquad + \frac{r u}{(1 + b u A^*)^2} + q_3 + q_4 + \mu 
           \Bigg] \\
       &\quad + \frac{A^*}{S^*} \theta 
           + 2 \gamma \frac{A^*}{C^*} \left( \frac{S^*}{C^* + S^*} \right)^2 
           + \frac{A^*}{R^*} \left( \alpha \lambda 
           + 2 \beta \gamma \left( \frac{S^*}{C^* + S^*} \right)^2 \right) \\
       &\leq \frac{\dot{A}}{A} - (q_2 + q_3 + \mu),
\end{cases}
\end{align*}

\begin{align*}
\begin{cases}
h_3(t) &= b_{33} + \sum_{i=1, i \neq 3}^{4} |b_{3i}| \\
h_3(t )&= b_{33} + |b_{31}| + |b_{32}| + |b_{34}| \\
       &= \frac{\dot{C}}{C} - \Bigg[ 
           2\beta \left( \frac{A^*}{A^* + S^*} \right)^2 
           + 2\beta \gamma \left( \frac{C^*}{C^* + S^*} \right)^2 
           + q_2 + q_3 + \mu 
           \Bigg] \\
       &\quad + \frac{C^*}{S^*} \frac{r u}{(1 + b u A^*)^2} 
           + \frac{C^*}{A^*} \sigma 
           + 2\beta \frac{C^*}{R^*} \left( \frac{S^*}{S^* + A^*} \right)^2 \\
       &\leq \frac{\dot{C}}{C} - (q_2 + q_3 + \mu),
\end{cases}       
\end{align*}

similarly,
\begin{align*}
\begin{cases}
h_4(t) &= b_{44} + \sum_{i=1}^{3} |b_{4i}| \\
h_4(t) &= b_{44} + |b_{41}| + |b_{42}| + |b_{43}| \\
       &= \frac{\dot{R}}{R} - \Bigg[ 
           -2\beta \left( \frac{S^*}{S^* + A^*} \right)^2 
           + \frac{r u}{(1 + b u A^*)^2} + q_2 + q_4 + \mu 
           \Bigg] \\
       &\quad + \frac{R^*}{S^*} \delta 
           + 2 \beta \frac{R^*}{C^*} \Bigg[ 
           \left( \frac{A^*}{A^* + S^*} \right)^2 
           + \gamma \left( \frac{C^*}{C^* + S^*} \right)^2 
           \Bigg] \\
       &\leq \frac{\dot{R}}{R} - (q_1 + q_4 + \mu).
\end{cases}       
\end{align*}

Let \( (b_1, b_2, b_3, b_4) \) be a vector in \( \mathbb{R}^4 \). The Lozinski measure \( \mu(B) \) is defined as \( \mu(B) = h_i \), \( i = 1, 2, 3, 4 \). Upon the integration of the Lozinski measure \( \mu(B) \), we apply limits as \( t \to \infty \), yielding the following sets of equations:

\begin{align*}
    \limsup_{t \to \infty} sup{\frac{1}{t}} \int_{0}^{t} h_1(t) dt & \leq \frac{1}{t} log \frac{S(t)}{S(0)} - (q_2 + q_3 + q_4) \\
    & < - (q_2 + q_3 + q_4),
\end{align*}

\begin{align*}
    \limsup_{t \to \infty} sup \frac{1}{t}\int_0^t h_2(t) dt & \leq \frac{1}{t} log\frac{A(t)}{A(0)} - (q_3 + q_4 + \mu) \\ 
    & < -(q_3 + q_4 + \mu), 
\end{align*}

\begin{align*}
    \limsup_{t \to \infty} sup{\frac{1}{t}} \int_{0}^{t} h_3(t) dt & \leq \frac{1}{t} log \frac{C(t)}{C(0)} - (q_2 + q_3 + \mu) \\
    & < - (q_2 + q_3 + \mu),
\end{align*}

\begin{align*}
    \limsup_{t \to \infty} sup \frac{1}{t}\int_0^t h_4(t) dt & \leq \frac{1}{t} log\frac{R(t)}{R(0)} - (q_2 + q_4 + \mu) \\ 
    & < -(q_2 + q_4 + \mu).
\end{align*}

From the above discussions, we can conclude that 
\begin{align*}
     \bar{q} = \limsup_{t \to \infty} sup \frac{1}{t}\int_0^t \mu(B) dt < 0.
\end{align*}
Thus, the endemic equilibrium point (EE), $E^*$, is deemed to be globally asymptotically stable.
\end{proof}


\section{Optimal Control Problem}
Optimal control techniques offer powerful mathematical tools for designing strategies to control the spread of infectious diseases, such as HBV (Hepatitis B Virus), within a community. In this study, we employ optimal control theory \cite{beay2022dynamical,ghosh2019qualitative,GUMEL2003409,khan2017transmission,zaman2008stability} to develop an effective strategy for reducing HBV infection rates across the population. Our primary objectives are to augment the count of recovered persons, \( R(t) \), while concurrently diminishing the susceptible population \( S(t) \) and the individuals afflicted with acute \( A(t) \) and chronic \( C(t) \) hepatitis B. To achieve these objectives, we utilize time-dependent control variables: vaccination \( u_1(t) \) for susceptible individuals and treatment \( u_2(t) \) for the infected population. This can be formalized with a cost functional that penalizes infection prevalence and intervention costs, optimized over time to determine the most effective deployment of resources.

The objective function is given by:

\[
J(u_1, u_2) = \int_{0}^{T} \{k_1 S + k_2 C + \frac{1}{2}(w_1 u_1^2 + w_2 u_2^2)\} dt
\]
subject to the control system:

\begin{align*}
\frac{dS}{dt} &= \lambda(1-\alpha C) - \beta \frac{2AS}{A+S} - \beta \gamma \frac{2CS}{C+S} - (u_1 + \mu)S \\
\frac{dA}{dt} &= \beta \frac{2AS}{A+S} + \beta \gamma \frac{2CS}{C+S} - \frac{ruA}{1+buA} - (\sigma + \mu)A \\
\frac{dC}{dt} &= \sigma A + (\alpha \lambda - u_2 - \tau - \mu)C \\
\frac{dR}{dt} &= u_1 S + \frac{ruA}{1+buA} + u_2 C - \mu R
\end{align*}
\text{The initial conditions are:}
\begin{equation*}
0 \leq S(0), \quad 0 \leq A(0), \quad 0 \leq C(0), \quad 0 \leq R(0).
\end{equation*}

In the objective functional, \(k_1\) and \(k_2\) represent the weight constants for the susceptible and chronic populations, respectively; \(w_1\) and \(w_2\) are the weight constants for vaccination and treatment, respectively.

\vspace{0.5cm}

\textbf{Goal}: To obtain an optimal pair $(u_1^*, u_2^*)$ such that $J(u_1^*, u_2^*) = \min J(u_1, u_2)$ where

\[
\mathcal{U} = \{(u_1, u_2) \in L^1(0, T) | (u_1(t), u_2(t)) \in [0, u_{1max}] \times [0, u_{2max}], \forall t \in [0, T] \}
\]

\vspace{0.5cm}
Using the control variables, Pontryagin's maximum principle \cite{zaman2008stability} transforms the system into an optimal control problem by introducing the following Hamiltonian function to be minimized:

\[
H(t, x, U, \Lambda) = f(t, x, U) + \Lambda g(t, x, U),
\]
which is equal to

\begin{align*}
H &= k_1 S + k_2 C + \frac{1}{2} w_1 u_1^2 + \frac{1}{2} w_2 u_2^2 + \Lambda_1 \left[ \lambda(1 - \alpha C ) -  \beta \frac{2AS}{A + S} -  \beta \gamma\frac{2CS}{C + S} - (u_1 + \mu) S \right] \\
&\quad + \Lambda_2 \left[ \beta\frac{2AS}{A + S} + \beta \gamma\frac{2CS}{C + S} - \frac{r u A}{1 + b u A} - (\sigma + \mu) A \right] \\
&\quad + \Lambda_3 \left[ \sigma A + (\alpha \lambda - u_2 - \tau - \mu) C \right] \\
&\quad + \Lambda_4 \left[u_1 S + \frac{r u A}{1 + b u A} + u_2 C - \mu R \right],
\end{align*}

where \( \Lambda_i, i = 1,2,3,4 \) are the Lagrangian multipliers of the optimization problem. It should satisfy Pontryagin's Maximum Principle \cite{zaman2008stability} as follows:
\[
\frac{d\Lambda_1}{dt} = -\frac{\partial H}{\partial S}, \quad \frac{d\Lambda_2}{dt} = -\frac{\partial H}{\partial A}, \quad \frac{d\Lambda_3}{dt} = -\frac{\partial H}{\partial C}, \quad \frac{d\Lambda_4}{dt} = -\frac{\partial H}{\partial R}
\]
with the transversality condition
\[
\Lambda_1(T) = 0, \quad \Lambda_2(T) = 0, \quad \Lambda_3(T) = 0, \quad \Lambda_4(T) = 0.
\]

The optimal conditions(co-state(adjoint) equations) are given by:
\[
\frac{\partial H}{\partial u_1} = 0, \quad \frac{\partial H}{\partial u_2} = 0 \quad.
\]
Now, the co-state equations are:
\begin{align*}
\frac{d\Lambda_1}{dt} &= \Lambda_1 \left(2 \beta \left( \frac{A}{A+S} \right)^2 + 2 \beta \gamma \left( \frac{C}{C+S} \right)^2 - (u_1 + \mu) \right) \\
&\quad - \Lambda_2 \left( 2 \beta \left( \frac{A}{A+S} \right)^2 + 2 \beta \gamma  \left( \frac{C}{C+S} \right)^2 \right) - \Lambda_4 u_1 - k_1, \\[10pt]
\frac{d\Lambda_2}{dt} &= \Lambda_1 2 \beta \left(  \frac{S}{A+S} \right)^2 + \Lambda_2 \left(\frac{r u}{(1 + b u A)^2} + (\sigma + \mu) -2 \beta \left(\frac{S}{A+S} \right)^2 \right) - \Lambda_3 \sigma - \Lambda_4 \frac{r u}{(1 + b u A)^2}, \\[10pt]
\frac{d\Lambda_3}{dt} &= \Lambda_1 \left(\alpha \lambda + 2 \beta \gamma \left(\frac{S}{C+S} \right)^2 \right) - \Lambda_2 2\beta \gamma \left(\frac{S}{C+S} \right)^2 +\Lambda_3 (u_2 +\tau +\mu - \alpha \lambda) - \Lambda_4 u_2 - k_2,  \\
&\quad - \Lambda_4 u_2 - k_2, \\[10pt]
\frac{d\Lambda_4}{dt} &= \Lambda_4 \mu
\end{align*}

Again, we write our Hamiltonian function as
\[
H = \frac{1}{2} w_1 u_1^2 + \frac{1}{2} w_2 u_2^2 - \Lambda_1 u_1 S - \Lambda_3 u_2 C + \Lambda_4 u_1 S + \Lambda_4 u_2 C + \text{terms free from } u_1, u_2
\]

Utilizing the optimality conditions \( \frac{\partial H}{\partial u_1} = 0 \), \( \frac{\partial H}{\partial u_2} = 0 \), we get:
\[
w_1 u_1 - \Lambda_1 S + \Lambda_4 S = 0
\]
and
\[
w_2 u_2 - \Lambda_3 C + \Lambda_4 C = 0
\]

which implies that
\[u_1 = \frac{(\Lambda_1 - \Lambda_4) S}{w_1}\]
and \quad
\[u_2 = \frac{(\Lambda_3 - \Lambda_4) C}{w_2}\]

So, the optimal controls required can be written in terms of the maximum and minimum values:
\[
u_1^* = \max \left\{ \min \left( \frac{1}{w_1} S(\Lambda_1 - \Lambda_2), 0 \right), 1 \right\}
\]
and
\[ u_2^* = \max \left\{ \min \left( \frac{1}{w_2} C(\Lambda_3 - \Lambda_4), 0 \right), 1 \right\}. \]

\section{Numerical Simulation for the Mathematical Model}
\subsection{Global Sensitivity Analysis}
Global Sensitivity Analysis (GSA) assesses how variations in model inputs influence model outputs throughout the entire range of possible input values. In epidemiological models, GSA helps identify the parameters that most significantly affect disease dynamics, such as transmission rates, recovery rates, and initial population sizes. This understanding is crucial for model calibration, interpretation of uncertainties, and prioritizing data collection. GSA provides a robust analysis of the consistency of the behavior of the model across different parameter values, improving the reliability of the model and quantifying the contribution of parameter uncertainty to the overall uncertainty in model predictions. For this analysis, we applied Latin Hypercube Sampling (LHS) together with the Partial Rank Correlation Coefficient (PRCC). Latin Hypercube Sampling (LHS) is a statistical method for efficiently generating a sample of parameter values from a multidimensional parameter space \cite{alrabaiah2020optimal, beay2022dynamical}. LHS ensures that the entire range of each parameter is evenly sampled, thereby enhancing the sample's representativeness and improving efficiency over simple random sampling. The PRCC method is used to measure how strongly each parameter is associated with the model output, as well as the direction of that influence. A positive PRCC indicates a direct relationship between the parameter and the output, whereas a negative PRCC suggests an inverse relationship. A magnitude close to 1 indicates a strong relationship, while values near zero indicate little or no effect. PRCC is especially useful for nonlinear and nonmonotonic models, as it uses ranks rather than raw values, making it less sensitive to nonlinear relationships. PRCC values can be plotted over time to analyze how parameters influence changes across different stages of disease progression, providing insights into which parameters are most impactful at each stage. We employed the PRCC technique to demonstrate the strong impact of model parameters on the reproduction number, $\mathcal{R}_o$. The results of the sensitivity analysis are illustrated in Figures \ref{bar_prcc} and \ref{line_prcc}, and the corresponding PRCC values with their p-values are listed in Table \ref{tab_PRCC_table}. From Figures~\ref{bar_prcc} and \ref{line_prcc}, it is clear that the parameter $\beta$ has the strongest influence on the model output, as it shows the highest positive PRCC value. 

The parameters $\alpha$ and $\lambda$ also display notable positive correlations, indicating that increases in these parameters tend to raise $\mathcal{R}_0$. This implies that reducing these parameters through appropriate intervention strategies can help lower infection levels.

In contrast, the parameter $\tau$ has the most significant negative PRCC value, followed by $\mu$ and $\theta$. These negative correlations suggest that increasing these parameters contributes to a reduction in $\mathcal{R}_0$ and, therefore, can help suppress the spread of the infection.

\vspace{2.4 cm}

\begin{table}[h]
\caption{PRCC values and corresponding p-values for $\mathcal{R}_0$ with respect to the model parameters} \label{tab_PRCC_table}
\centering
\begin{tabular}{|l|l|l|} \hline
Parameter & PRCC Value & p-value \\ \hline\hline
$\beta$ & 0.3916 & 0.0001 \\ \hline
$\sigma$ & -0.0099 & 0.9223 \\ \hline
$\gamma$ & -0.0144 & 0.8867 \\ \hline
$\theta$ & -0.0472 & 0.6404 \\ \hline
$\tau$ & -0.8720 & 0.0000 \\ \hline
$\mu$ & -0.0577 & 0.5682 \\ \hline
$\alpha$ & 0.3528 & 0.0003 \\ \hline
$\lambda$ & 0.2574 & 0.0099 \\ \hline
$r$ & 0.1102 & 0.2745 \\ \hline
$u$ & -0.0253 & 0.8024 \\ \hline
\end{tabular}
\end{table}

\begin{figure}
 \centering
    \begin{subfigure}
        \centering
        \includegraphics[scale=1]{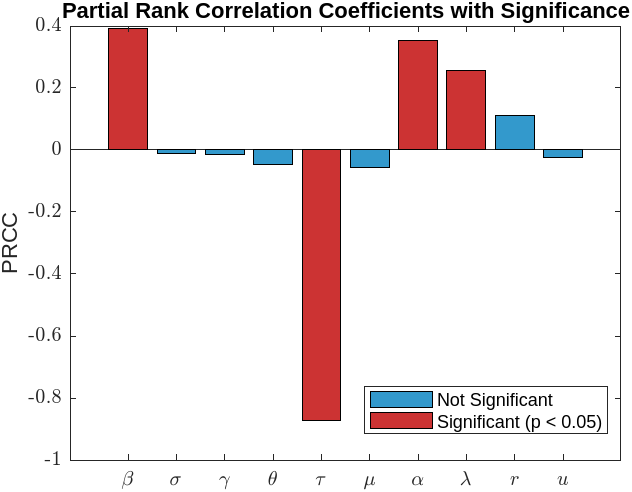}
        \caption{Bar plot showing the PRCC values of $\mathcal{R}_0$ for the model parameters used in the analysis}
        \label{bar_prcc}
    \end{subfigure}%
    
\vspace{0.5cm}    
    \begin{subfigure}
        \centering
        \includegraphics[scale=1]{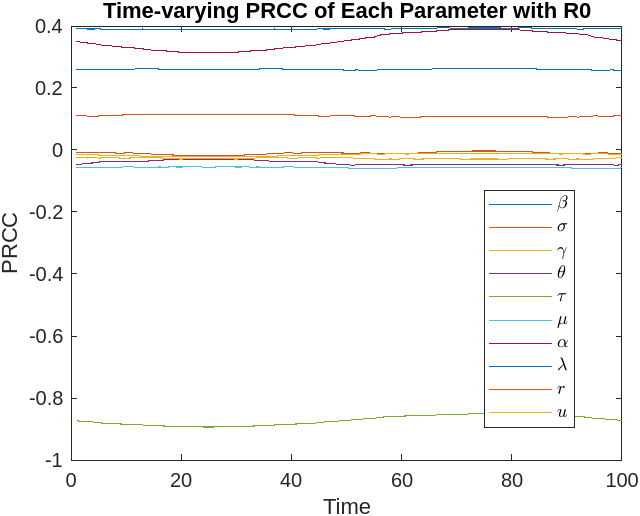}
        \caption{Time-varying PRCC of each parameters involved in $\mathcal{R}_0$}
        \label{line_prcc}
    \end{subfigure}%

\end{figure}


In light of the above findings, we have generated 2D contour \& 3D surface profiles of significant parameters against $\mathcal{R}_0$.

\begin{figure}[!h]
  \centering
  \begin{minipage}{.5\textwidth}
    \centering
    \includegraphics[width=\linewidth]{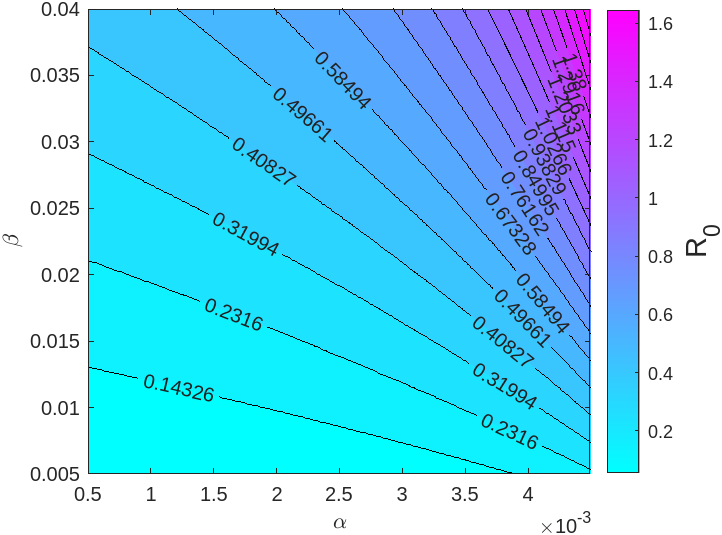}
    \includegraphics[width=\linewidth]{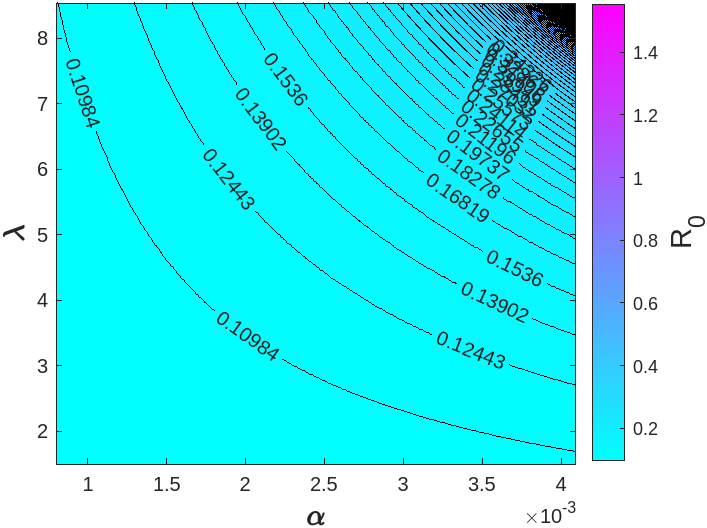}
    \includegraphics[width=\linewidth]{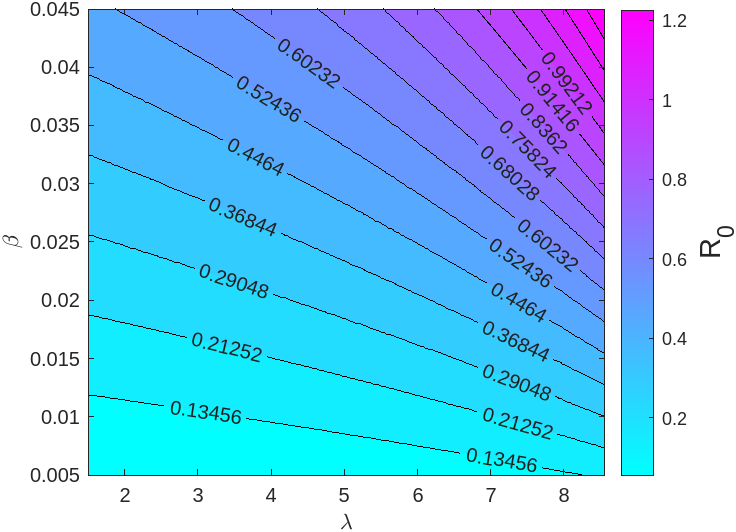}
  \end{minipage}%
  \begin{minipage}{.5\textwidth}
    \centering
    \includegraphics[width=\linewidth]{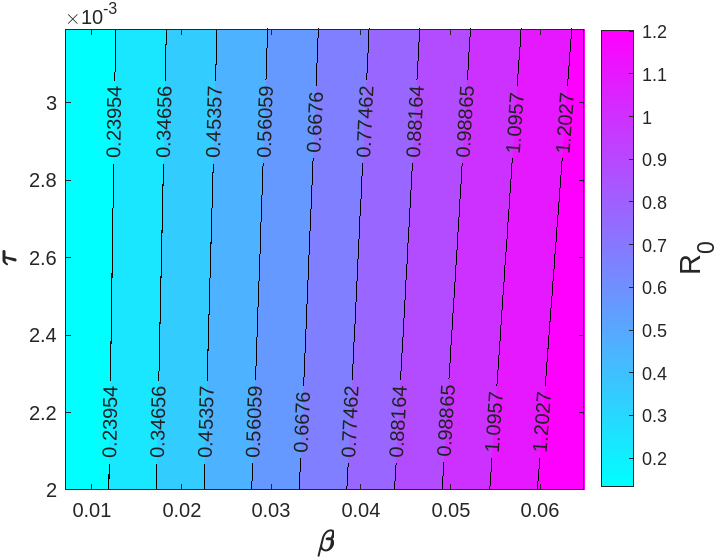}
    \includegraphics[width=\linewidth]{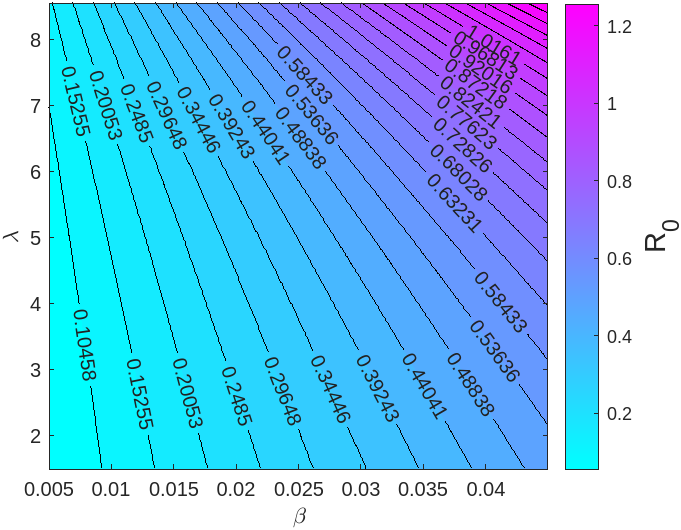}
    
  \end{minipage}
  \caption{2D contour profile of $R_{0}$ versus significant parameters $\alpha$, $\beta$, $\gamma$ and $\tau$.}
  \label{contour}
\end{figure}
			
\newpage		
\begin{figure}[!h]
  \centering
  \begin{minipage}{.5\textwidth}
    \centering
    \includegraphics[width=\linewidth]{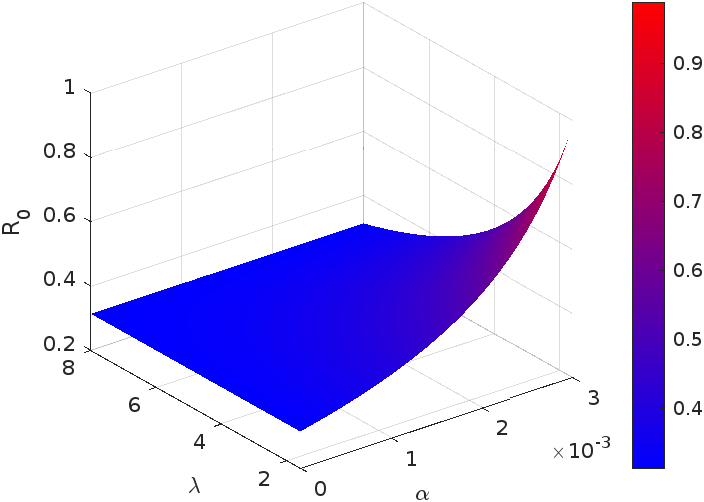}
    \includegraphics[width=\linewidth]{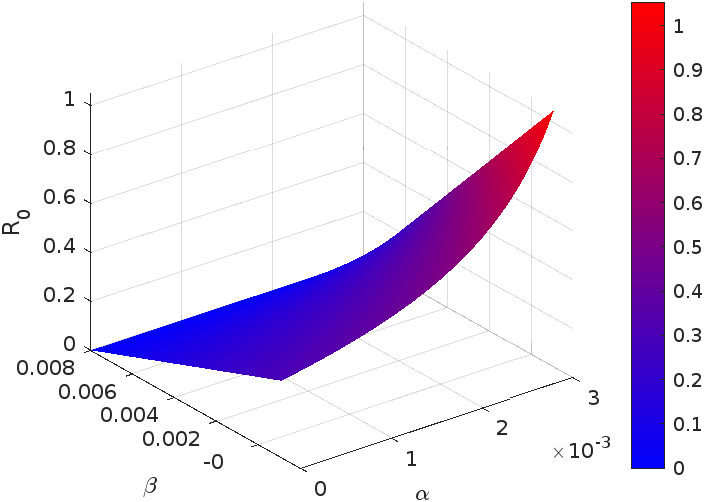}
    \includegraphics[width=\linewidth]{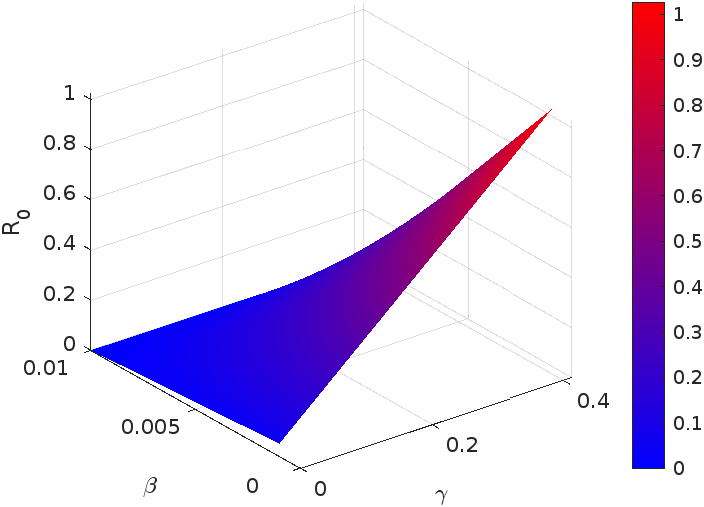}
  \end{minipage}%
  \begin{minipage}{.5\textwidth}
    \centering
    \includegraphics[width=\linewidth]{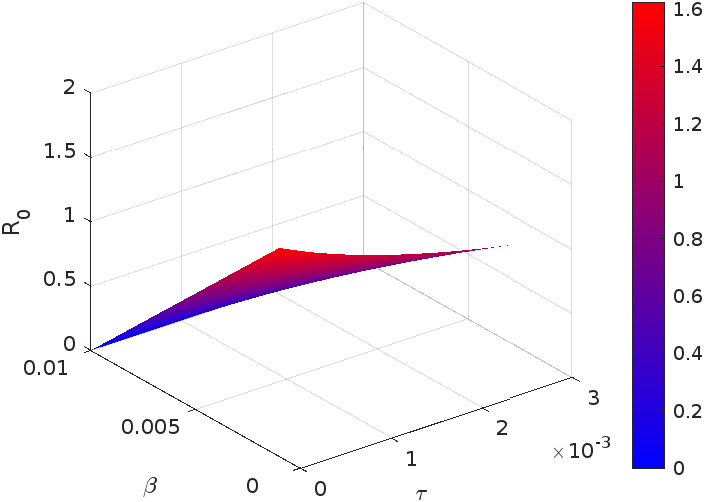}
    \includegraphics[width=\linewidth]{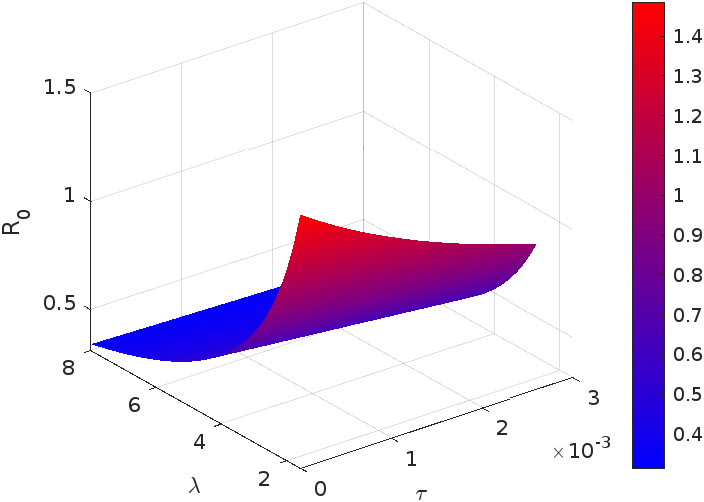}
  \end{minipage}
  \caption{3D surface profile of $R_{0}$ versus significant parameters $\alpha$, $\beta$, $\gamma$ and $\tau$.}
  \label{surface}
\end{figure}

To understand how Hepatitis B might either die out or remain a long-term issue in a population, we used two different sets of parameters—Set A and Set B. Set A, with parameter set \{$\lambda = 6.56$, $\alpha = 0.003$, $\beta = 0.0095$, $\gamma = 0.42$, $\delta = 0.02$, $\mu = 0.0221$, $\sigma = 0.33$, $\theta = 0.012$, $\tau = 0.0026$, r = 1.5, u = 0.39, b = 0.5\}, was chosen to represent conditions that could lead to the extinction of the disease. Meanwhile, Set B, with parameter values \{$\lambda = 6.56$, $\alpha = 0.003$, $\beta = 0.0091$, $\gamma = 0.42$, $\delta = 0.02$, $\mu = 0.0121$, $\sigma = 0.33$, $\theta = 0.012$, $\tau = 0.0026$, r = 0.09005, u = 0.19, b = 0.007\}, reflects scenarios where the disease is more likely to persist in the community.

Both sets of parameters were simulated with the same initial population states: \((S(0), A(0), C(0), R(0)) = (1000, 320, 50, 10)\), subjecting the process to run for 800 time units, using MATLAB's ode45 solver. The results are illustrated in Fig. \ref{fig_hbv_r0_01} and Fig. \ref{fig_R0 under time}, which describe the temporal dynamics of the susceptible, acute, chronic, and recovered populations. 

In Fig.\ref{fig_hbv_r0_01}, the blue trajectories corresponding to Set A (where $R_0 < 1$), illustrate a consistent reduction in the infection. With time passing, the acute and chronic compartments reach zero, with the susceptible and recovered populations dominating in the long term. This direction is suggestive of effective disease eradication under the given conditions. 

Conversely, the red trajectories, computed using Set B ($R_0 > 1$), illustrate recurrent levels of acute and chronic infection over the long term, suggestive of the repeated occurrence of the disease. Here, every compartment, i.e., recovered, chronic, acute, and susceptible, has nonzero equilibrium values, reflecting the long-term presence of infected cases in the population.

Together, these findings point to a foundational epidemiological principle: if $R_0 < 1$, the disease is fated for eventual eradication with a residual population of only recovered and susceptible members. If $R_0 > 1$, the infection will persist, leaving behind a residual burden on the population in the form of both acute and chronic diseases.

We also examined the time-dependent reproduction number, \( \mathcal{R}_0(t) \), for each intervention strategy; see Fig. \ref{fig_R0 under time}. When both vaccination and treatment controls were implemented in tandem, \( \mathcal{R}_0(t) \) remained below the critical threshold of one for a significant portion of the time horizon. This is a promising outcome, indicating that such a strategy could feasibly lead to long-term control or even elimination of the virus.

\begin{figure}[!htbp]
    \centering
    \includegraphics[width=7in]{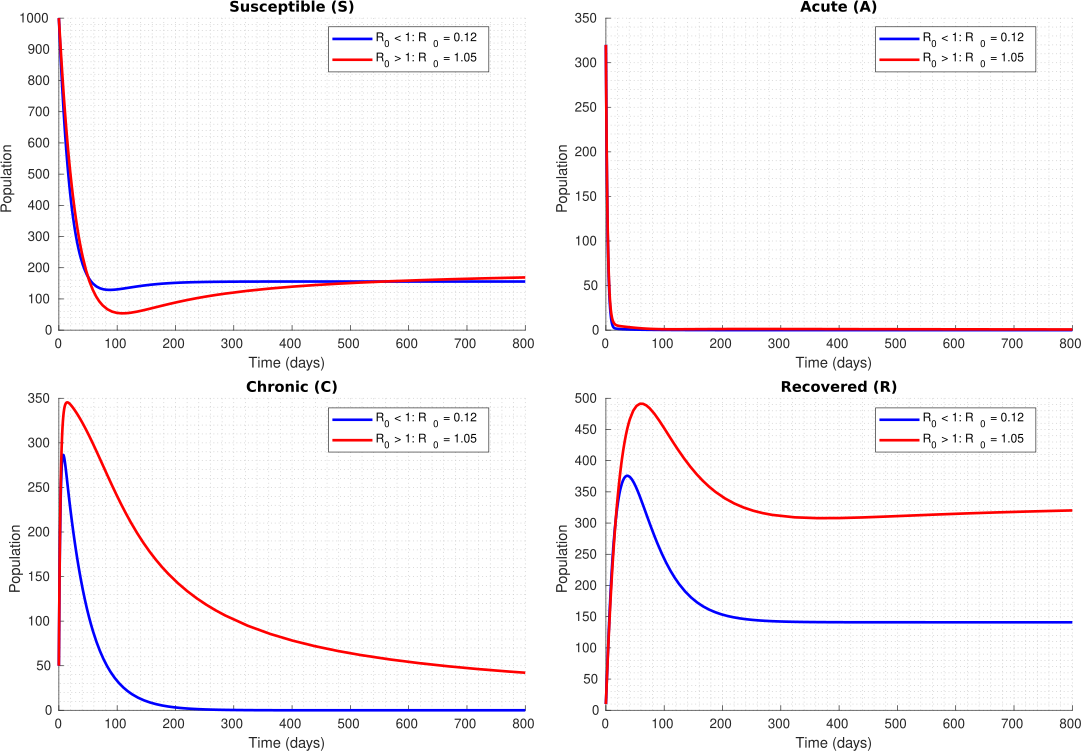}
    \caption{Comparison of the HBV transmission dynamics across all compartments under two scenarios: 
    \(\mathcal{R}_0 > 1\) (Uncontrolled) and \(\mathcal{R}_0 < 1\) (Controlled). 
    Shaded backgrounds indicate regions of high (red) and low (blue) transmission potential. Results suggest effective suppression of infection 
    when \(\mathcal{R}_0 < 1\) through targeted parameter control.}
    \label{fig_hbv_r0_01}
\end{figure}

\begin{figure}[!h]
    \centering
    \includegraphics[width=1\linewidth]{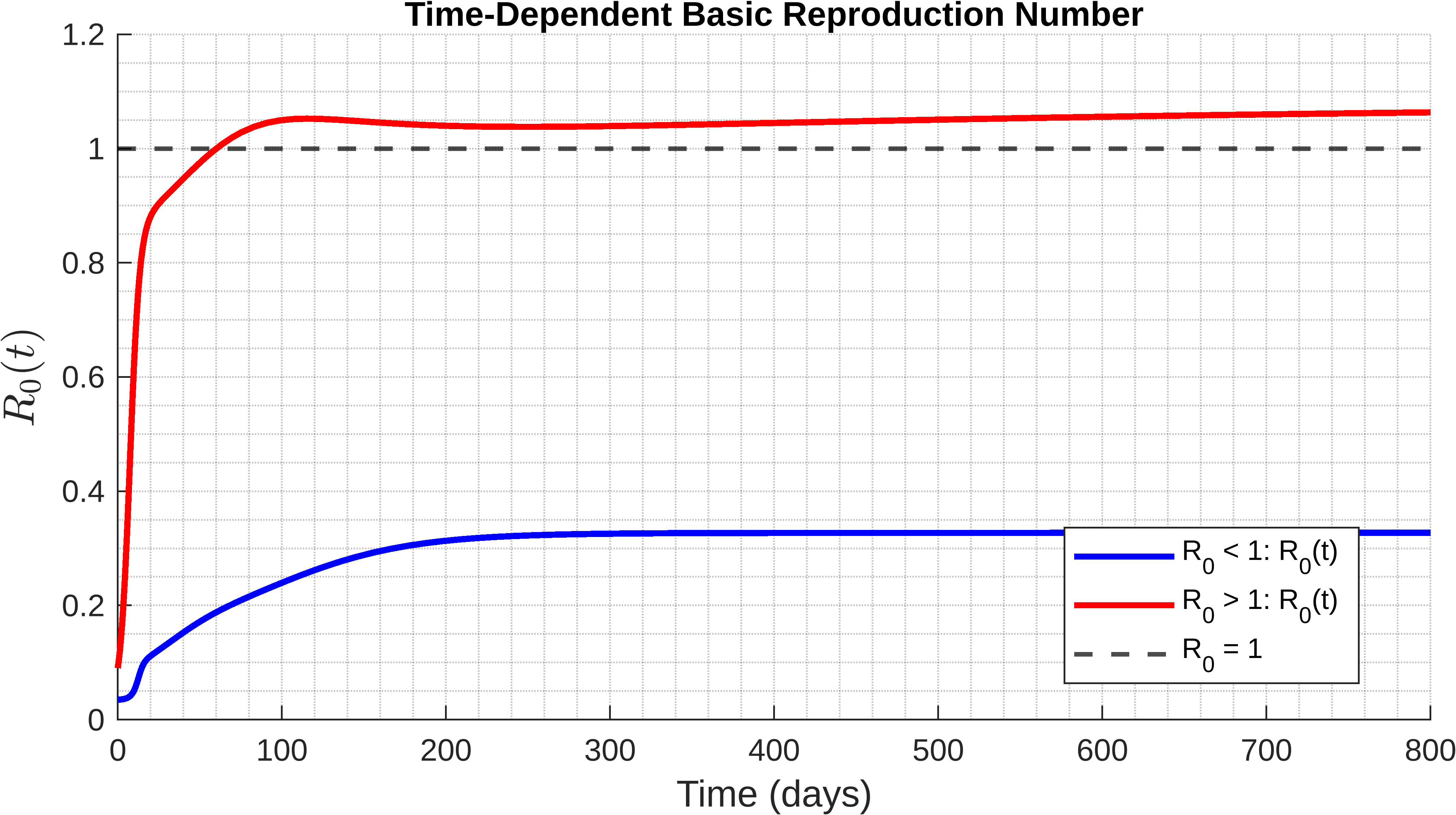}
    \caption{Evolution of reproduction number with time for the parameter sets A and B.}
    \label{fig_R0 under time}
\end{figure}

\subsection{Simulation of Optimal Control Analysis}
We carry out numerical simulations of the proposed HBV model, both with and without the control strategies \cite{alrabaiah2020optimal} to illustrate their impact. To obtain the solutions, we use the forward-backward sweep method, which is a standard numerical approach for solving optimal control problems. In this scheme, the state equations are solved forward in time using the control functions, and the adjoint equations are solved backward in time to update the co-state variables. The control values \( u_1(t) \) and \( u_2(t) \) are updated by minimizing the Hamiltonian at each time step. Optimal control is achieved when the control profile converges according to Pontryagin’s principle. The weight constants and balancing constants are considered as $k_1$ = 1, $k_2$ = 10, $w_1$ = 15, and $w_2$ = 5. In the following figures, the red dashed lines illustrate the dynamics of different population compartments under control measures, whilst the blue lines signify population behavior in the absence of interventions. We graphically show the influence and efficacy of the control measures proposed by different sets of strategies. By incorporating these two control variables, we investigated three strategies for eradicating HBV infection. The impact and efficacy of each strategy for eliminating the disease are depicted graphically.

\section*{Scenario 1: Use of Single Control Parameters}
\subsubsection*{Strategy 1: Implementing Vaccination (\( u_1(t) \)) Control Only}
In this strategy, we use a single control parameter to curb the transmission of HBV and observe its effect on the population. So we use the control parameter \( u_1(t) \) as an optimized parameter and \( u_2(t) \) as an unoptimized parameter and compare the model dynamics when there is no control at all. In this situation, we observe a significant decrease in infections in both acute and chronic populations; see Fig. \ref{fig_u1_optimal}. 

\begin{figure}[!htbp]
    \centering
    \includegraphics[width=\textwidth]{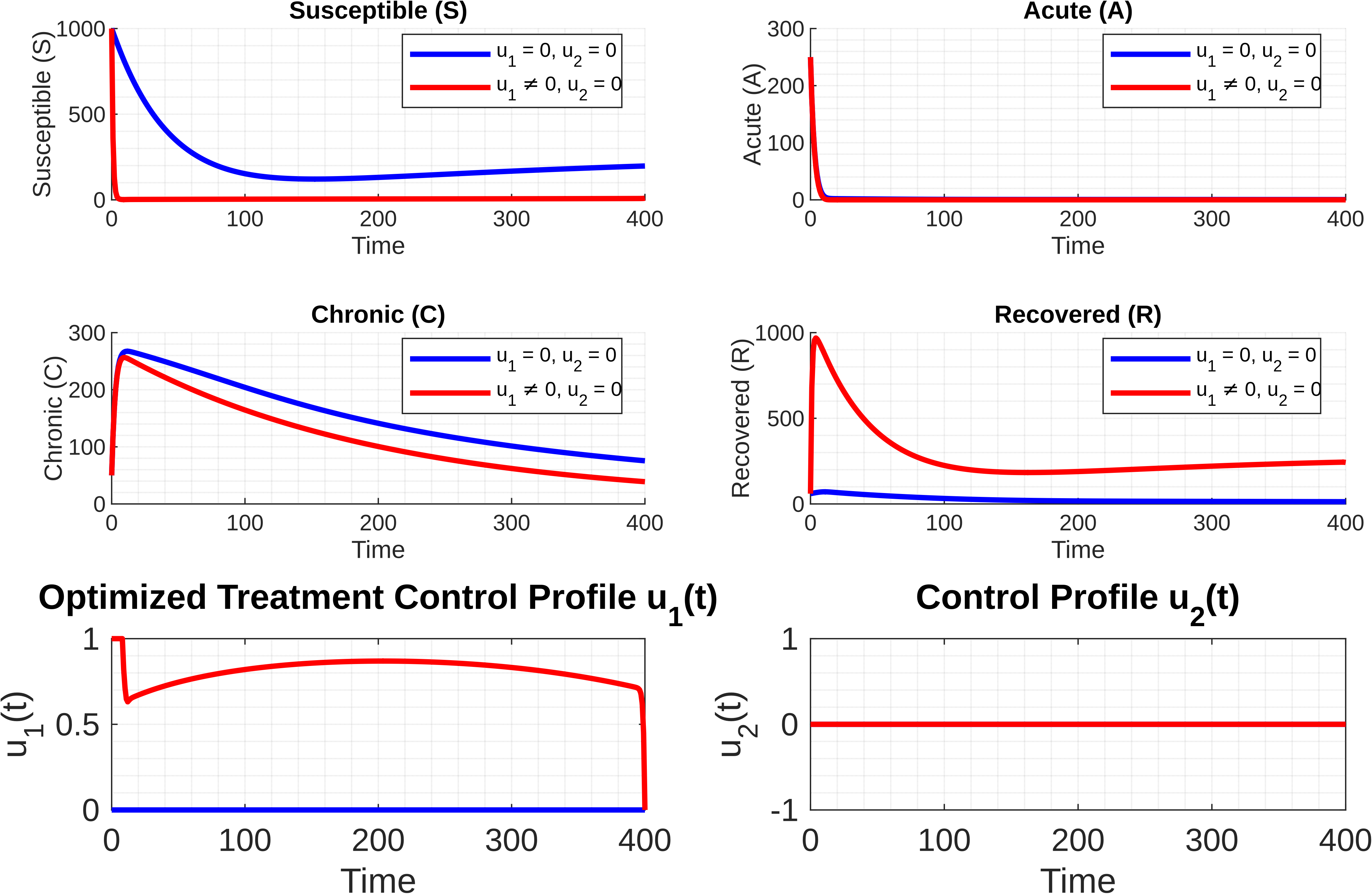}
    \caption{Population and control dynamics for HBV model under optimal control strategy with $u_1 \neq 0$ and \( u_2(t) = 0\). The plots show Susceptible population $S(t)$, Acute population $A(t)$, Chronic population $C(t)$, Recovered population $R(t)$, and the control profiles \( u_1(t) \) and \( u_2(t) \), where \( u_1(t) \) is optimized.}
    \label{fig_u1_optimal}
\end{figure}

\subsection*{Strategy 2: Implementing Treatment (\( u_2(t) \)) Control Only}In this strategy, we again use one control parameter to curb HBV transmission and observe the effect on the population. So we use the control parameter \( u_2(t) \) as an optimized parameter and \( u_1(t) \) as an unoptimized parameter and compare the model dynamics when there is no control at all. In this case, we observe a significant decrease in infections in both acute and chronic populations; see Fig. \ref{fig_u2_optimal}.

\begin{figure}[!htbp]
    \centering
    \includegraphics[width=\textwidth]{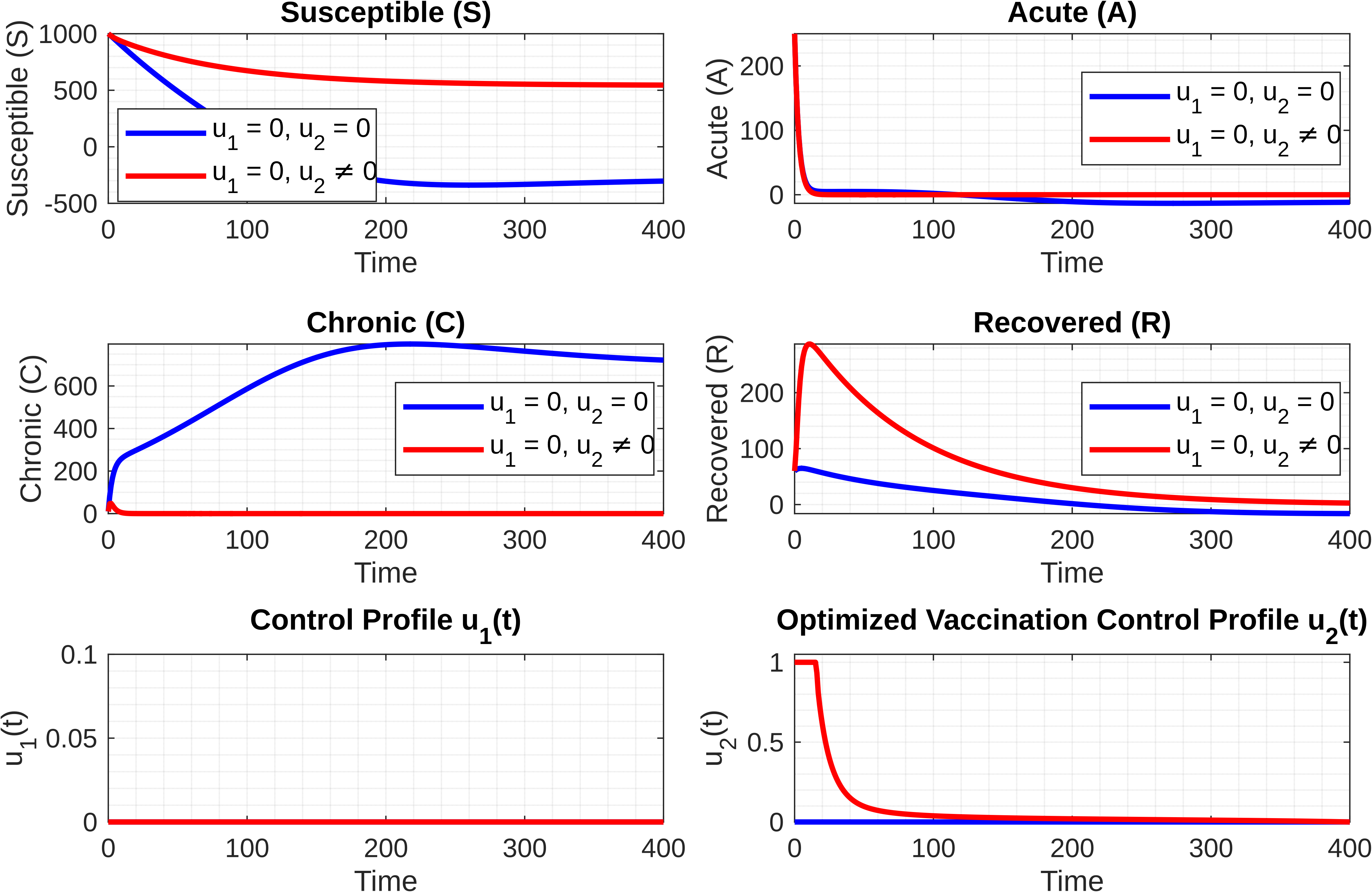}
    \caption{Population and control dynamics for HBV model under optimal control strategy with $u_2 \neq 0$ and \( u_1(t) = 0 \). The plots show (a) Susceptible population $S(t)$, Acute population $A(t)$, Chronic population $C(t)$, Recovered population $R(t)$, and the control profiles \( u_1(t) \) and \( u_2(t) \), where \( u_2(t) \) is optimized.}
    \label{fig_u2_optimal}
\end{figure}

\section*{Scenario 2: Use of Coupled Control Parameters}
\subsection*{Strategy 3: Implementing both Vaccination (\( u_1(t) \)) \& Treatment (\( u_2(t) \)) Controls}In this strategy, we use two control parameters to control the transmission of HBV and see the influence on population dynamics. So we use the control parameters \( u_1(t) \) and \( u_2(t) \) both as optimized parameters and compare the model dynamics when there is no control at all. We get very good results in controlling infection in both acute and chronic populations, see Fig. \ref{fig_u1_u2_optimal}.

\begin{figure}[!htbp]
    \centering
    \includegraphics[width=\textwidth]{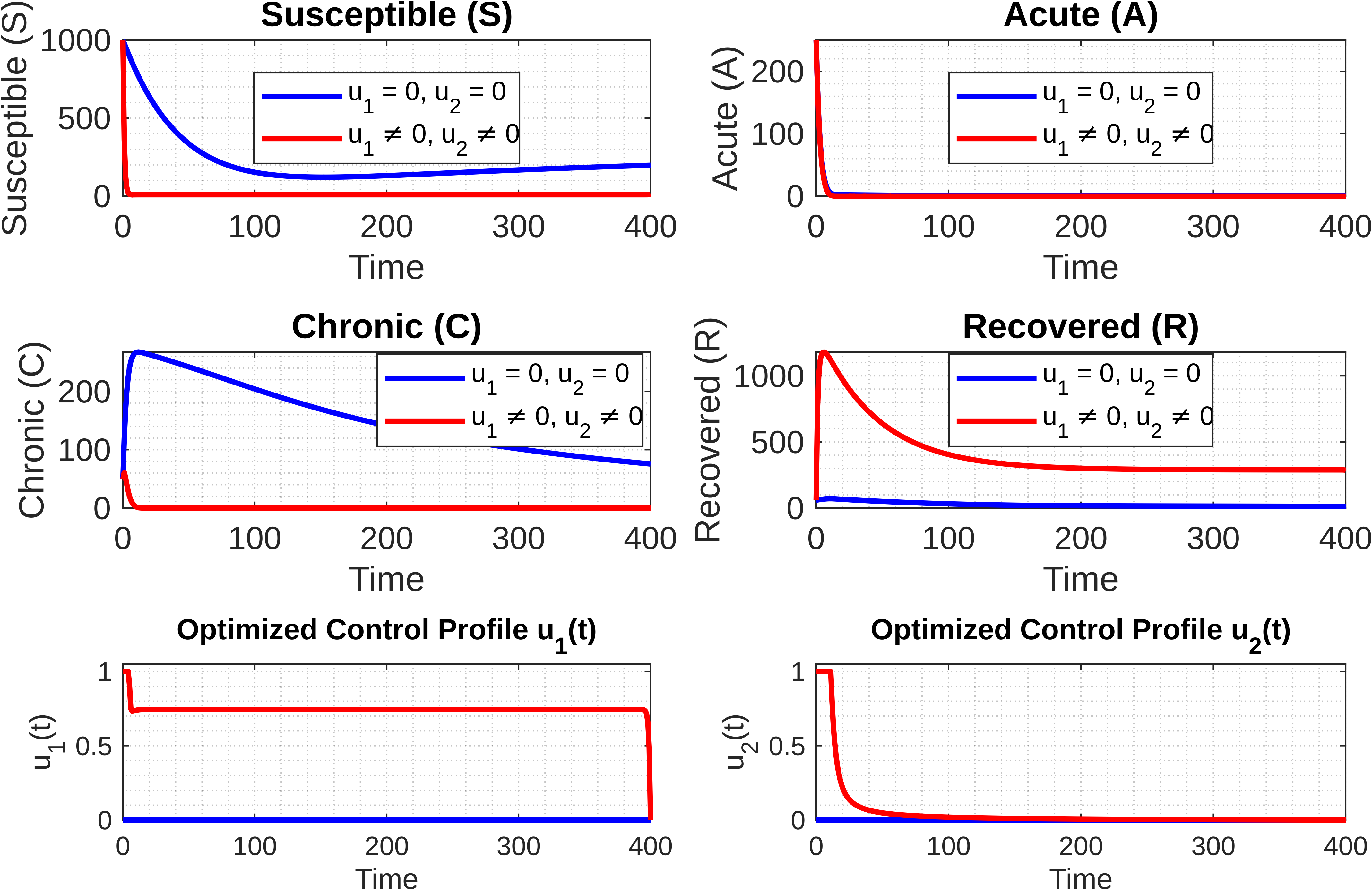}
    \caption{Population and control dynamics for HBV model under optimal control strategy with $u_1 \neq 0$ and $u_2 \neq 0$. The plots show Susceptible population $S(t)$, Acute population $A(t)$, Chronic population $C(t)$, Recovered population $R(t)$, and the optimized control profiles of \( u_2(t) \) and \( u_2(t) \).}
    \label{fig_u1_u2_optimal}
\end{figure}

\section*{Results and Discussion}
Our findings underscore the substantial impact that optimal control measures can have on the spread of Hepatitis B. In the absence of any intervention, both acute and chronic infections remained consistently elevated throughout the simulation period. Introducing a control strategy focused solely on reducing susceptibility through \( u_1(t) \), representing vaccination, led to a noticeable reduction in the number of susceptible individuals, alongside a modest increase in the recovered population. This suggests that preventive strategies alone can begin to alter the trajectory of the epidemic.

In contrast, applying control exclusively to the chronically infected population via \( u_2(t) \), representing treatment, resulted in a more immediate decline in chronic infections. However, this approach had a limited effect on the susceptible and recovered groups, indicating a narrower but more targeted outcome. The most significant reduction in the overall disease burden was achieved by employing both control strategies simultaneously. Under this combined approach, the infected compartments declined sharply while the recovered population rose considerably, highlighting the importance of integrated prevention and treatment efforts.

During the early phase of the simulation, control efforts were applied with the greatest intensity, suggesting that prompt intervention is critical for limiting transmission. As infection levels declined over time, the optimal control inputs decreased gradually. This trend mirrors real-world public health strategies, where early and aggressive responses are essential to suppress outbreaks.

Additionally, as illustrated in Fig. \ref{fig_hbv_strategy_comparison}, implementing either vaccination control \( u_1(t) \) or treatment control \( u_2(t) \) individually led to a notable decline in the infected population. But when we were implementing both the control measures simultaneously, the reduction in the infectious population was surprisingly very high, and the recovered population was greater than any other single intervention. Moreover, we observed reduced peak infection levels and a smaller final outbreak size across all compartments, reinforcing the effectiveness of a combined intervention approach. 

\begin{figure}[!htbp]
    \centering
    \includegraphics[width=7in]{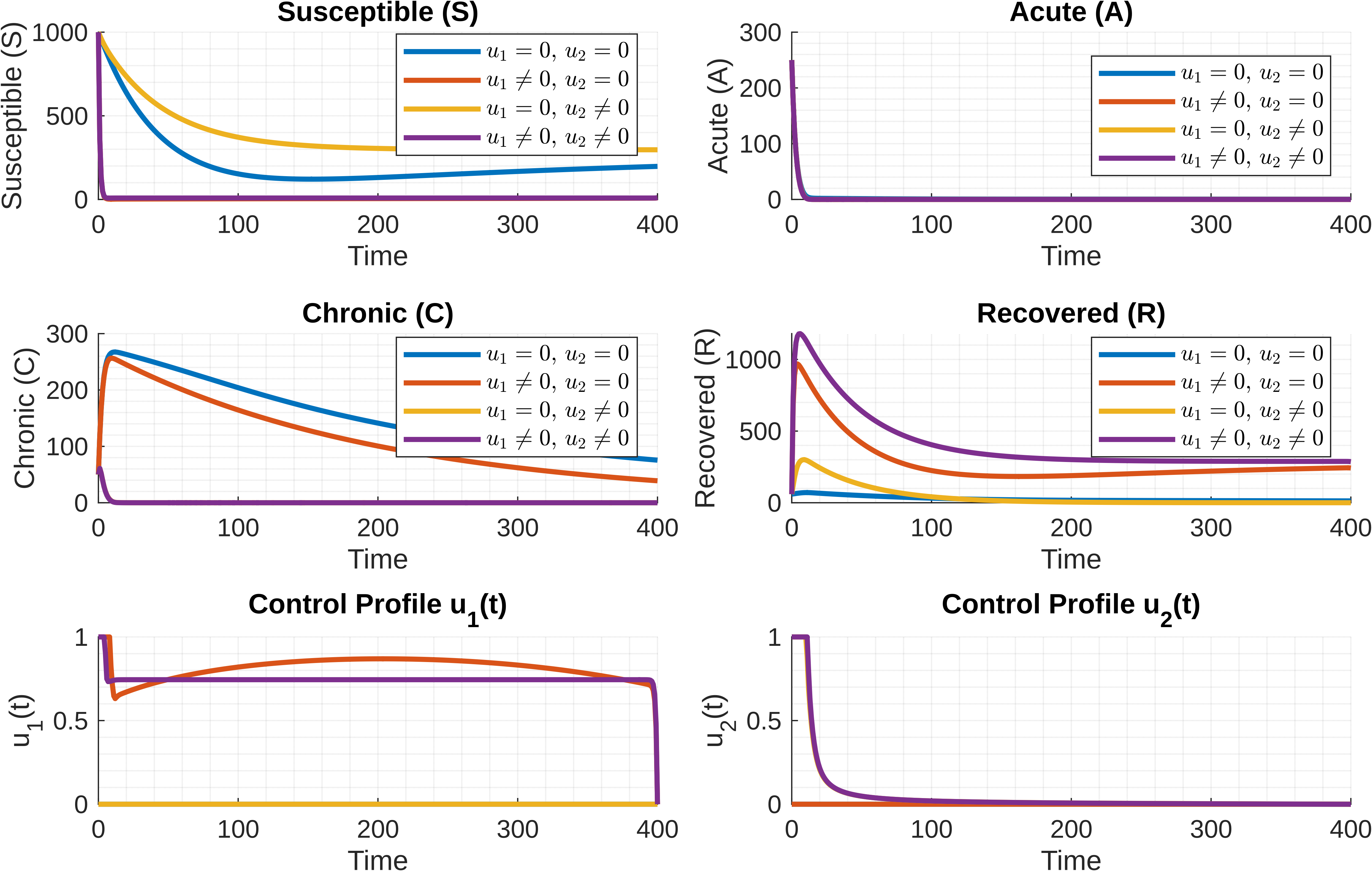}
    \caption{Comparative analysis of HBV model under four optimal control strategies: no control (\( u_1(t) \), \( u_2(t) \)), \( u_1(t) \) optimized only, \( u_2(t) \) optimized only, and both \( u_1(t) \), \( u_2(t) \) optimized. The first four sub-figures show the evolution of the susceptible, acute, chronic, and recovered populations, respectively, while the last two graphs represent the optimized control functions \( u_1(t) \) and \( u_2(t) \) over time.}
    \label{fig_hbv_strategy_comparison}
\end{figure}

\section{Conclusion and Recommendations}
In this work, we constructed and analyzed a mathematical model to study the transmission dynamics of Hepatitis B virus (HBV) infection, incorporating the roles of vaccination and treatment interventions. The model accounts for key epidemiological processes such as vertical transmission, chronic infection, and nonlinear treatment effects. Through rigorous analytical methods, we derived the basic reproduction number, \( \mathcal{R}_0 \), and demonstrated that the disease-free equilibrium is both locally and globally asymptotically stable when \( \mathcal{R}_0 < 1 \). Conversely, the infection persists when \( \mathcal{R}_0 > 1 \), in line with expectations from classical epidemic theory \cite{who2024global}.

A global sensitivity analysis employing Latin Hypercube Sampling (LHS) and Partial Rank Correlation Coefficients (PRCC) identified the transmission rates (\( \alpha \) \& \( \beta \)) and the disease-induced death rate (\( \tau \)) as the most impactful parameters influencing HBV prevalence. These findings suggest that effective strategies focused on reducing contact transmission and improving patient survival through timely treatment could substantially limit disease spread.

To evaluate the impact of intervention strategies, we applied Pontryagin’s Maximum Principle \cite{pontryagin1962} and derived the optimality conditions for control measures involving vaccination and treatment. Numerical simulations confirmed that combining both controls is more effective than applying either strategy in isolation, significantly reducing the infected population and increasing recovery rates.

\subsection*{Recommendations}

Based on the theoretical and computational results, we offer the following recommendations for HBV control and public health decision-making:

\begin{enumerate}
    \item \textbf{Promote Integrated Vaccination and Treatment Campaigns} \\
    The model underscores the effectiveness of jointly administering vaccination and treatment. Policy frameworks should support programs that ensure coordinated delivery of both interventions, maximizing their synergistic effect.

    \item \textbf{Reduce Contact Transmission in High-Risk Populations} \\
    Given the high sensitivity of the model to the contact rates \( \alpha \) and \( \beta \), public health strategies should emphasize behavioral change, awareness campaigns, and targeted interventions in populations at increased risk of exposure \cite{shepard2006}.

    \item \textbf{Expand and Sustain Vaccination Programs} \\
    The significant role of vaccination in lowering \( \mathcal{R}_0 \) justifies continued investment in expanding vaccine coverage, particularly in newborns and unvaccinated adults\cite{al2024global}. Immunization programs should include outreach components for hard-to-reach and vulnerable groups.

    \item \textbf{Improve Access to Early Diagnosis and Treatment} \\
    To effectively reduce disease-induced mortality (\( \tau \)) and enhance recovery, early diagnosis and equitable access to antiviral therapy must be prioritized, especially in settings with limited healthcare infrastructure \cite{terrault2018}.

    \item \textbf{Incorporate Mathematical Modeling into Public Health Policy} \\
    This study illustrates the utility of optimal control and epidemiological modeling in informing the timing and allocation of interventions. Public health agencies are encouraged to integrate model-based tools into HBV elimination strategies\cite{easterbrook20242024}.

    \item \textbf{Invest in Surveillance and Model Refinement} \\
    Accurate and ongoing data collection is essential for calibrating models and assessing interventions. Enhancing health data infrastructure and supporting surveillance systems will facilitate more adaptive and responsive policymaking.
\end{enumerate}

The current model forms a foundational structure that can be extended in future studies to include additional epidemiological complexities such as age stratification, co-infection dynamics, spatial heterogeneity, and behavioral changes. These extensions are essential to better inform policy and support long-term HBV elimination goals at national and global levels\cite{who2024global}.

\section*{Authors contributions} All authors have contributed equally.

\section*{Data availability} This manuscript has no associated data.

\section*{Declarations}
\section*{Conflict of interests} There is no conflict of interest regarding this work.

\bibliography{references}

@article{khan2018spreading,
  title={Spreading dynamic of acute and carrier hepatitis B with nonlinear incidence},
  author={Khan, Tahir and Zaman, Gul and Saleh Alshomrani, Ali},
  journal={PloS one},
  volume={13},
  number={4},
  pages={e0191914},
  year={2018},
  publisher={Public Library of Science San Francisco, CA USA}
}

@article{li1996geometric,
  title={A geometric approach to global-stability problems},
  author={Li, Michael Y and Muldowney, James S},
  journal={SIAM Journal on Mathematical Analysis},
  volume={27},
  number={4},
  pages={1070--1083},
  year={1996},
  publisher={SIAM}
}

@article{zaman2008stability,
  title={Stability analysis and optimal vaccination of an SIR epidemic model},
  author={Zaman, Gul and Kang, Yong Han and Jung, Il Hyo},
  journal={BioSystems},
  volume={93},
  number={3},
  pages={240--249},
  year={2008},
  publisher={Elsevier}
}

@article{zarin2021dynamics,
  title={Dynamics of five grade leishmania epidemic model using fractional operator with Mittag--Leffler kernel},
  author={Zarin, Rahat and Khan, Amir and Inc, Mustafa and Humphries, Usa Wannasingha and Karite, Touria},
  journal={Chaos, Solitons \& Fractals},
  volume={147},
  pages={110985},
  year={2021},
  publisher={Elsevier}
}

@article{khan2016classification,
  title={Classification of different hepatitis B infected individuals with saturated incidence rate},
  author={Khan, Tahir and Zaman, Gul},
  journal={SpringerPlus},
  volume={5},
  number={1},
  pages={1082},
  year={2016},
  publisher={Springer}
}

@article{khan2017transmission,
  title={The transmission dynamic and optimal control of acute and chronic hepatitis B},
  author={Khan, Tahir and Zaman, Gul and Chohan, M Ikhlaq},
  journal={Journal of biological dynamics},
  volume={11},
  number={1},
  pages={172--189},
  year={2017},
  publisher={Taylor \& Francis}
}

@article{van2002reproduction,
  title={Reproduction numbers and sub-threshold endemic equilibria for compartmental models of disease transmission},
  author={Van den Driessche, Pauline and Watmough, James},
  journal={Mathematical biosciences},
  volume={180},
  number={1-2},
  pages={29--48},
  year={2002},
  publisher={Elsevier}
}

@article{khan2020stability,
  title={Stability analysis of leishmania epidemic model with harmonic mean type incidence rate},
  author={Khan, Amir and Zarin, Rahat and Inc, Mustafa and Zaman, Gul and Almohsen, Bandar},
  journal={The European Physical Journal Plus},
  volume={135},
  number={6},
  pages={528},
  year={2020},
  publisher={Springer}
}

@book{castillo2002mathematical,
  title={Mathematical approaches for emerging and reemerging infectious diseases: models, methods, and theory},
  author={Castillo-Chavez, Carlos and Blower, Sally and van den Driessche, Pauline and Kirschner, Denise and Yakubu, Abdul-Aziz},
  volume={126},
  year={2002},
  publisher={Springer Science \& Business Media}
}

@article{martin1974logarithmic,
  title={Logarithmic norms and projections applied to linear differential systems},
  author={Martin Jr, Robert H},
  journal={Journal of Mathematical Analysis and Applications},
  volume={45},
  number={2},
  pages={432--454},
  year={1974},
  publisher={Academic Press}
}

@article{yadav2023study,
  title={Study of an SIQR model with optimal control techniques: A mathematical approach},
  author={Yadav, Sudha and Bhadauria, Archana Singh and Verma, Vijai Shanker},
  journal={Results in Control and Optimization},
  volume={13},
  pages={100327},
  year={2023},
  publisher={Elsevier}
}

@article{umdekar2023seir,
  title={An SEIR model with modified saturated incidence rate and Holling type II treatment function},
  author={Umdekar, Shilpa and Sharma, Praveen Kumar and Sharma, Shivram},
  journal={Computational and Mathematical Biophysics},
  volume={11},
  number={1},
  pages={20220146},
  year={2023},
  publisher={De Gruyter}
}

@article{beay2022dynamical,
  title={Dynamical analysis of a modified epidemic model with saturated incidence rate and incomplete treatment},
  author={Beay, Lazarus Kalvein and Anggriani, Nursanti},
  journal={Axioms},
  volume={11},
  number={6},
  pages={256},
  year={2022},
  publisher={MDPI}
}

@article{ghosh2019qualitative,
  title={Qualitative analysis and optimal control strategy of an SIR model with saturated incidence and treatment},
  author={Ghosh, Jayanta Kumar and Ghosh, Uttam and Biswas, MHA and Sarkar, Susmita},
  journal={Differential Equations and Dynamical Systems},
  pages={1--15},
  year={2019},
  publisher={Springer}
}

@article{hethcote2000mathematics,
  title={The mathematics of infectious diseases},
  author={Hethcote, Herbert W},
  journal={SIAM review},
  volume={42},
  number={4},
  pages={599--653},
  year={2000},
  publisher={SIAM}
}

@article{castillo2002computation,
  title={On the computation of r. And its role on global stability carlos castillo-chavez∗, zhilan feng, and wenzhang huang},
  author={Castillo-Chavez, Carlos},
  journal={Mathematical approaches for emerging and reemerging infectious diseases: an introduction},
  volume={1},
  pages={229},
  year={2002}
}

@book{diekmann2000mathematical,
  title={Mathematical epidemiology of infectious diseases: model building, analysis and interpretation},
  author={Diekmann, Odo and Heesterbeek, Johan Andre Peter},
  volume={5},
  year={2000},
  publisher={John Wiley \& Sons}
}

@article{muldowney1990compound,
  title={Compound matrices and ordinary differential equations},
  author={Muldowney, James S},
  journal={The Rocky Mountain Journal of Mathematics},
  pages={857--872},
  year={1990},
  publisher={JSTOR}
}

@article{li1995global,
  title={Global stability for the SEIR model in epidemiology},
  author={Li, Michael Y and Muldowney, James S},
  journal={Mathematical biosciences},
  volume={125},
  number={2},
  pages={155--164},
  year={1995},
  publisher={Elsevier}
}

@article{GUMEL2003409,
title = {A qualitative study of a vaccination model with non-linear incidence},
journal = {Applied Mathematics and Computation},
volume = {143},
number = {2},
pages = {409-419},
year = {2003},
issn = {0096-3003},
doi = {https://doi.org/10.1016/S0096-3003(02)00372-7},
url = {https://www.sciencedirect.com/science/article/pii/S0096300302003727},
author = {A.B. Gumel and S.M. Moghadas}
}

@misc{who2024global,
  title={Global Hepatitis Report 2024: Action for Access in Low-and Middle-Income Countries},
  author={WHO},
  year={2024},
  publisher={World Health Organization Geneva, Switzerland}
}

@article{liu2022numerical,
  title={Numerical dynamics and fractional modeling of hepatitis B virus model with non-singular and non-local kernels},
  author={Liu, Peijiang and Din, Anwarud and Zarin, Rahat},
  journal={Results in Physics},
  volume={39},
  pages={105757},
  year={2022},
  publisher={Elsevier}
}

@article{zarin2021fractional,
  title={Fractional modeling and optimal control analysis of rabies virus under the convex incidence rate},
  author={Zarin, Rahat and Ahmed, Iftikhar and Kumam, Poom and Zeb, Anwar and Din, Anwarud},
  journal={Results in Physics},
  volume={28},
  pages={104665},
  year={2021},
  publisher={Elsevier}
}

@article{thornley2008hepatitis,
  title={Hepatitis B in a high prevalence New Zealand population: a mathematical model applied to infection control policy},
  author={Thornley, Simon and Bullen, Chris and Roberts, Mick},
  journal={Journal of Theoretical Biology},
  volume={254},
  number={3},
  pages={599--603},
  year={2008},
  publisher={Elsevier}
}

@article{zou2010modeling,
  title={Modeling the transmission dynamics and control of hepatitis B virus in China},
  author={Zou, Lan and Zhang, Weinian and Ruan, Shigui},
  journal={Journal of theoretical biology},
  volume={262},
  number={2},
  pages={330--338},
  year={2010},
  publisher={Elsevier}
}

@article{li2002qualitative,
  title={Qualitative analyses of SIS epidemic model with vaccination and varying total population size},
  author={Li, Jianquan and Ma, Zhien},
  journal={Mathematical and Computer Modelling},
  volume={35},
  number={11-12},
  pages={1235--1243},
  year={2002},
  publisher={Elsevier}
}

@article{khan2021modeling,
  title={Modeling and sensitivity analysis of HBV epidemic model with convex incidence rate},
  author={Khan, Amir and Zarin, Rahat and Hussain, Ghulam and Usman, Auwalu Hamisu and Humphries, Usa Wannasingha and Gomez-Aguilar, JF},
  journal={Results in Physics},
  volume={22},
  pages={103836},
  year={2021},
  publisher={Elsevier}
}

@article{diekmann1990definition,
  title={On the definition and the computation of the basic reproduction ratio R 0 in models for infectious diseases in heterogeneous populations},
  author={Diekmann, Odo and Heesterbeek, Johan Andre Peter and Metz, Johan Anton Jacob},
  journal={Journal of mathematical biology},
  volume={28},
  pages={365--382},
  year={1990},
  publisher={Springer}
}

@article{alrabaiah2020optimal,
  title={Optimal control analysis of hepatitis B virus with treatment and vaccination},
  author={Alrabaiah, Hussam and Safi, Mohammad A and DarAssi, Mahmoud H and Al-Hdaibat, Bashir and Ullah, Saif and Khan, Muhammad Altaf and Shah, Syed Azhar Ali},
  journal={Results in Physics},
  volume={19},
  pages={103599},
  year={2020},
  publisher={Elsevier}
}

@article{zhang2008backward,
  title={Backward bifurcation of an epidemic model with saturated treatment function},
  author={Zhang, Xu and Liu, Xianning},
  journal={Journal of mathematical analysis and applications},
  volume={348},
  number={1},
  pages={433--443},
  year={2008},
  publisher={Elsevier}
}

@article{wang2006backward,
  title={Backward bifurcation of an epidemic model with treatment},
  author={Wang, Wendi},
  journal={Mathematical biosciences},
  volume={201},
  number={1-2},
  pages={58--71},
  year={2006},
  publisher={Elsevier}
}

@article{wang2004bifurcations,
  title={Bifurcations in an epidemic model with constant removal rate of the infectives},
  author={Wang, Wendi and Ruan, Shigui},
  journal={Journal of Mathematical Analysis and Applications},
  volume={291},
  number={2},
  pages={775--793},
  year={2004},
  publisher={Elsevier}
}

@book{pontryagin1962,
  title={The Mathematical Theory of Optimal Processes},
  author={Pontryagin, L.S. and others},
  year={1962},
  publisher={Pergamon Press}
}

@article{shepard2006,
  author={Shepard, C.W. and Simard, E.P. and Finelli, L. and others},
  title={Hepatitis B virus infection: epidemiology and vaccination},
  journal={Epidemiologic Reviews},
  year={2006},
  volume={28},
  pages={112--125}
}

@article{terrault2018,
  author={Terrault, N.A. and Lok, A.S. and McMahon, B.J. and others},
  title={Update on prevention, diagnosis, and treatment of chronic hepatitis B: AASLD 2018 Hepatitis B Guidance},
  journal={Hepatology},
  year={2018},
  volume={67},
  number={4},
  pages={1560--1599}
}

@article{easterbrook20242024,
  title={WHO 2024 hepatitis B guidelines: an opportunity to transform care},
  author={Easterbrook, Philippa J and Luhmann, Niklas and Bajis, Sahar and Min, Myat Sandi and Newman, Morkor and Lesi, Olufunmilayo and Doherty, Meg C},
  journal={The lancet Gastroenterology \& hepatology},
  volume={9},
  number={6},
  pages={493--495},
  year={2024},
  publisher={Elsevier}
}

@article{al2024global,
  title={Global perspectives on the hepatitis B vaccination: challenges, achievements, and the road to elimination by 2030},
  author={Al-Busafi, Said A and Alwassief, Ahmed},
  journal={Vaccines},
  volume={12},
  number={3},
  pages={288},
  year={2024},
  publisher={MDPI}
}
\bibliographystyle{acm}

\end{document}